\documentclass[11pt]{article}
\usepackage{xcolor}
\usepackage{graphicx}
\usepackage{amsmath,amsfonts,amssymb,graphics,amsthm}
\usepackage{hyperref}
\usepackage{enumerate}
\usepackage{bbm}
\usepackage{mathrsfs}

\numberwithin{equation}{section}
\newtheorem{theorem}{Theorem}[section]
\newtheorem{corollary}[theorem]{Corollary}
\newtheorem{lemma}[theorem]{Lemma}
\newtheorem{proposition}[theorem]{Proposition}

\newtheorem{remark}[theorem]{Remark}
\newtheorem{definition}[theorem]{Definition}
\theoremstyle{remark}

\newcommand{\wt}{\widetilde}
\newcommand{\wh}{\widehat}
\def \eps {\varepsilon}

\newcommand{\R}{\mathbb{R}}
\newcommand{\A}{\mathbb{A}}
\newcommand{\D}{\mathbb{D}}
\newcommand{\C}{\mathbb{C}}
\newcommand{\cC}{\mathcal{C}}
\newcommand{\cD}{\mathcal{D}}
\newcommand{\E}{\mathbb{E}}
\newcommand{\bE}{{\bf E}}
\renewcommand{\P}{\mathbb{P}}
\newcommand{\bP}{{\bf P}}
\renewcommand{\S}{\mathbb{S}}

\newcommand{\dist}{\mathrm{dist}}
\DeclareMathOperator{\SLE}{SLE}

\DeclareMathOperator{\CLE}{CLE}

\newcommand{\lp}{\mathrm{loop}}
\newcommand{\cW}{\mathcal{W}}
\newcommand{\cL}{\mathcal{L}}
\newcommand{\cQ}{\mathcal{Q}}
\newcommand{\cK}{\mathcal{K}}
\newcommand{\Fill}{\mathrm{Fill}}
\newcommand{\Cout}{\mathcal{C}^{\mathrm{out}}}

\begin{document}

\title{Uniqueness of generalized conformal restriction measures and Malliavin-Kontsevich-Suhov measures for $c \in (0,1]$}

\author{Gefei Cai\footnote{\url{caigefei1917@pku.edu.cn}; Beijing International Center for Mathematical Research, Peking University}\hspace{2cm}
Yifan Gao\footnote{\url{gaoyifan75@westlake.edu.cn}; Institute for Theoretical Sciences, Westlake University}}

\date{}

\maketitle

\begin{abstract}
In this paper, we present a unified approach to establish the uniqueness of generalized conformal restriction measures with central charge $c \in (0, 1]$ in both chordal and radial cases, 
by relating these measures to the Brownian loop soup.
Our method also applies to the uniqueness of the Malliavin-Kontsevich-Suhov loop measures for $c \in (0,1]$, which was recently obtained in [Baverez-Jego (2024)] for all $c \leq 1$ from a CFT framework of SLE loop measures. In contrast, though only valid for $c \in (0,1]$, our approach provides additional probabilistic insights, as it directly links natural quantities of MKS measures to loop-soup observables.
\bigskip

\end{abstract}

\section{Introduction}
In the study of Brownian intersection exponents~\cite{MR1796962}, Lawler and Werner first realized that any conformally invariant process satisfying a certain restriction property would give rise to the same intersection exponents as the Brownian motion.
This idea, combined with Schramm's discovery of SLE process \cite{Sc2000}, finally led to the rigorous determination of the Brownian intersection exponents~\cite{LSW2001a,LSW2001b,LSW2002a,LSW2002b}. To further deepen this idea,
Lawler, Schramm and Werner \cite{lawler2003conformal} gave a complete understanding of the \emph{conformal restriction measures} in the chordal case (which corresponds to Definition~\ref{def:chordal} with $c=0$ below). 
In particular, they showed that the boundary of conformal restriction samples can be described by variants of $\SLE_{8/3}$.
Conformal restriction measures in other cases are also studied in~\cite{wu2015conformal, qian2018conformal}.

For general $\SLE_\kappa$ with $\kappa\in (0,4]$, there is a corresponding parameter $c\le 1$ called \emph{central charge} related to it via 
\begin{equation}\label{eq:ckappa}
	c(\kappa)=1-6\bigg(\frac{2}{\sqrt{\kappa}}-\frac{\sqrt{\kappa}}{2}\bigg)^2,
\end{equation}
such that $\SLE_\kappa$ and its variants
satisfy a \emph{generalized conformal restriction} property that involves the Brownian loop mass with the corresponding $c$; see e.g. \cite{lawler2003conformal,MR2118865,MR2518970,werner2013cle,qian2021generalized}. For $c\in (0,1]$, corresponding to $\kappa\in (\frac{8}{3},4]$, there is another simple way to obtain a general restriction sample by taking the filled union of a standard one with all loop-soup clusters it intersects from an independent Brownian loop soup of intensity $\frac{c}{2}$.
However, it remains open whether these generalized restriction measures are unique.

There also exists a loop version of restriction measure, known as the \emph{Malliavin-Kontsevich-Suhov (MKS) measure}~\cite{kontsevich2007malliavin}. It is an infinite measure defined on simple loops on \(\C\) with a parameter \( c \leq 1 \); see Definition~\ref{def:MKS} below. To construct an MKS measure, one uses the loop version of \(\SLE_\kappa\)~\cite{kemppainen2016nested,zhan2021sle} (see also~\cite{werner2008conformally,benoist2016sle} for the special cases \( c = 0, -2 \)). 
Recently, Baverez and Jego~\cite{baverez2024cft} established the framework on the conformal field theory (CFT) of SLE, and as an application, they obtained the uniqueness of MKS measures for all $c\le 1$. It would be interesting to explore whether their approach can be adapted to show the uniqueness of generalized restriction measures.

In this paper, we present a unified and direct approach to establish the uniqueness of all these restriction measures for \( c \in (0,1] \), based on a key observation that connects these measures to the \emph{Brownian loop soup} or the \emph{$\SLE_{8/3}$ loop soup}. In particular, we explicitly express various natural quantities associated with these measures in terms of loop-soup observables, which provides further probabilistic understanding of these measures.

In the following, we will state our results on generalized restriction measures and MKS measures in Sections~\ref{sec:intro-chordal} and~\ref{sec:intro-loop}, respectively, and then provide several applications in Section~\ref{sec:implication}.

\subsection{Generalized chordal restriction measures}\label{sec:intro-chordal}

For convenience, we focus on the unit disk \(\D \subset \C\), with two marked boundary points \(-1\) and \(1\). Let \(\S^1 := \partial\D\). 
Let \(\mathcal{K}\) (resp.\ \(\mathcal{K}_-\)) be the collection of simply connected closed sets \(K \subset \overline{\D}\) such that \(K \cap \S^1 = \{-1, 1\}\) (resp.\ \(K \cap \S^1 = \S^1 \cap \{z : \Im z \geq 0\}\)). 
Let \(\mathcal{Q}\) be the collection of compact sets \(A \subset \overline{\D}\) such that \(\D \setminus A\) is simply connected and \(-1, 1 \notin A\). For each \(A \in \mathcal{Q}\), we choose a point $w_0\in \S^1\setminus A$ (depending on $A$) other than $-1,1$, and
let \(f_{A,w_0}: \D \setminus A \to \D\) be the conformal map that preserves \(-1,1\) and $w_0$. We also denote the set of \(A \in \mathcal{Q}\) with \(A \cap \S^1 \subset \{z \in \C : \Im z < 0\}\) by \(\mathcal{Q}_-\). 

We equip \(\mathcal{K}\) (resp.\ \(\mathcal{K}_-\)) with the \(\sigma\)-algebra generated by the collection of events \(\{K \in \mathcal{K} : K \cap A = \emptyset\}\) (resp.\ \(\{K \in \mathcal{K}_- : K \cap A = \emptyset\}\)) for all \(A \in \mathcal{Q}\) (resp.\ \(A \in \mathcal{Q}_-\)).

\begin{definition}\label{def:chordal}
	Let $c,\alpha\in\R$.
	We say that a probability measure $\P$ on $\mathcal{K}$ (resp.\ $\mathcal{K}_-$) satisfies the two-sided (resp.\ one-sided) chordal $c$-restriction with exponent $\alpha$, if for all $A\in\mathcal{Q}$ (resp.\ $A\in\mathcal{Q}_-$), we have
	\begin{equation}\label{eq:chordal-res}
		\frac{d\P(K)}{d\P_A(K)}{\bf 1}_{K\cap A=\emptyset}={\bf 1}_{K\cap A=\emptyset}(f_{A,w_0}'(1)f_{A,w_0}'(-1))^\alpha\exp\left(-\frac{c}{2}\Lambda_{\D}(K,A)\right),
	\end{equation}
	where $\P_A:=\P\circ f_{A,w_0}$, and $\Lambda_\D(K,A)$ is the total mass of  loops on $\D$ intersecting both $K$ and $A$ under the Brownian loop measure $\mu^{\rm BL}$, defined in \cite{lawler2004brownian}. Note that $f_{A,w_0}'(1)f_{A,w_0}'(-1)$ does not depend on the choice of $w_0$.
\end{definition}

When \( c = 0 \), Definition~\ref{def:chordal} reduces to the \emph{standard} chordal restriction measures in~\cite{lawler2003conformal}. The uniqueness of such measures is straightforward, as~\eqref{eq:chordal-res} now gives \(\P[K \cap A = \emptyset] = (f_{A,w_0}'(1) f_{A,w_0}'(-1))^\alpha\).

Note that the two-sided case above has been studied in \cite{qian2021generalized}. 
In \cite[Proposition 6.2]{qian2021generalized}, the author constructed measures that satisfy the two-sided chordal $c$-restriction with exponent \(\alpha\) in the range \( c \leq 1 \) (i.e., \(\kappa \in (0,4]\)) and \(\alpha \geq \frac{6-\kappa}{2\kappa}\), using variants of \(\SLE_\kappa\). However, it remains open whether these measures are unique or whether this range of \(\alpha\) is maximal for which these measures exist.

The first main result of this paper is to provide positive answers to both questions for \( c \in (0,1] \). Similar result also holds for the one-sided chordal $c$-restriction measures.

\begin{theorem}\label{thm:chordal-unique}
	Let $c\in(0,1]$ and $\kappa\in(\frac{8}{3},4]$ be related via \eqref{eq:ckappa}. Then the two-sided (resp.\ one-sided) chordal $c$-restriction with exponent $\alpha$ exists if and only if $\alpha\ge\frac{6-\kappa}{2\kappa}$ (resp.\ $\alpha>0$). Furthermore, when $\alpha\ge\frac{6-\kappa}{2\kappa}$ (resp.\ $\alpha>0$), there exists a unique probability measure on $\mathcal{K}$ (resp.\ $\mathcal{K}_-$) satisfying the two-sided (resp.\ one-sided) chordal $c$-restriction with exponent $\alpha$ in Definition~\ref{def:chordal}.
\end{theorem}

The main ingredient to show the uniqueness part of Theorem~\ref{thm:chordal-unique} is to express the non-intersection probability $\P[K\cap A=\emptyset]$ by some quantity only involving the Brownian loop soup. 
For this purpose, we let $\mathcal{L}_\D$ be the Brownian loop soup on $\D$ of intensity $\frac{c}{2}$, that is, a Poisson point process with intensity measure $\frac{c}{2}\mu^{\rm BL}$ restricted to loops fully contained in $\D$. 
For any bounded set $B\subset\C$, the \emph{filling} of $B$ is defined as the complement of unbounded connected component of $\C\setminus \overline{B}$ ($\overline{B}$ denotes the closure of $B$). Hence, the filling of any set is closed by our definition.
Let $\wt A$ be the filling of the union of $A$ and all the loop clusters in $\mathcal{L}_\D$ that intersect $A$. 
We need to consider the following event
\begin{equation}\label{eq:E}
E:=\{\text{$-1$ and $1$ are in the same connected component of $\overline{\D}\setminus \wt A$}\}.
\end{equation}
For $c\in(0,1]$, almost surely, no closure of loop cluster in $\mathcal{L}_\D$ intersects $\S^1$ by~\cite[Lemma~9.4]{sheffield2012conformal}, and the set of closures of loop clusters is locally finite by~\cite[Lemma~9.7]{sheffield2012conformal}. Hence, $\wt A\cap\S^1=A\cap\S^1$ for all $A\in\mathcal{Q}$, and $E$ happens a.s. when $A\in\mathcal{Q}_-$. Moreover,
on the event $E$, there is a non-empty open connected component of $\D\setminus \wt A$ that has $\pm1,w_0$ on its boundary (recall that $w_0$ is some fixed point on $\S^1\setminus A$), and we let $f_{\wt A,w_0}$ be the conformal map from this component to $\D$ fixing $\pm1$ and $w_0$. Then we have
\begin{theorem}\label{thm:chordal-intersect}
	Let $c\in(0,1]$ and $\alpha\in \R$. 
	Suppose $\P$ satisfies the two-sided (resp.\ one-sided) chordal $c$-restriction with exponent $\alpha$. Then for any $A\in\mathcal{Q}$ (resp.\ $A\in\mathcal{Q}_-$),
	\begin{equation}\label{eq:general}
		\P[K\cap A=\emptyset]=\bE\left[{\bf 1}_E\left(f_{\wt A,w_0}'(1)f_{\wt A,w_0}'(-1)\right)^\alpha\right]
	\end{equation}
	where $\bE$ denotes the expectation with respect to the Brownian loop soup $\mathcal{L}_\D$.
\end{theorem} 
The proof of \eqref{eq:general} is based on a simple observation that $\exp\left(-\frac{c}{2}\Lambda_{\D}(K,A)\right)$ in \eqref{eq:chordal-res} is in fact the probability that there is no loop in $\mathcal{L}_\D$ that intersects both $K$ and $A$ given $K$. Hence, by interpolating extra randomness from the Brownian loop soup, we obtain the first equation that $\P[K\cap A=\emptyset]=\bE\otimes\E_{A}\left[(f_{A,w_0}'(1)f_{A,w_0}'(-1))^\alpha {\bf 1}_{K\cap A_1=\emptyset}\right]$, where $\E_{A}$ represents the expectation for a restriction sample in $\D\setminus A$, and $A_1$ is the filling of the union of $A$ and all loops in $\mathcal{L}_\D$ that intersect $A$. Then, one can iterate the first equation until no loops can be added anymore, i.e. arriving at $\wt A$, and the limiting equation is just \eqref{eq:general}; see Section~\ref{subsec:chordal} for details. 
In fact, one can verify \eqref{eq:general} straightforwardly if $K$ is constructed by the union of a standard restriction sample with exponent $\alpha$ and all the clusters in $\mathcal{L}_\D$ that intersect it, which satisfies \eqref{eq:chordal-res} directly.

Moreover, similar results also hold for the generalized radial and trichordal restriction measures; see Section~\ref{sec:radial} for detailed discussions.
Next, we turn to the loop case.

\subsection{The Malliavin-Kontsevich-Suhov measure}\label{sec:intro-loop}

Let \(\mu_\C\) be a Borel measure on the set of simple loops on \(\C\). We say that \(\mu_\C\) is \textit{non-trivial} if for any bounded domain \(D\) and \(\delta > 0\), the $\mu_\C$-mass of simple loops \(\ell\) with \(\ell \subset D\) and \({\rm diam}(\ell) > \delta\) is positive and finite. 

\begin{definition}[{\cite{kontsevich2007malliavin}}]\label{def:MKS}
	A non-trivial measure \(\mu_\C\) on the set of simple loops on $\C$ is called a Malliavin-Kontsevich-Suhov measure if the following holds.
	For any simply connected domain \(D \subset \C\), define \(\mu_D\) by 
	\begin{equation}\label{eq:mks}
		\frac{d\mu_D(\eta)}{d\mu_\C(\eta)} \mathbf{1}_{\eta \subset D}:= \mathbf{1}_{\eta \subset D} \exp\left(\frac{c}{2} \Lambda^*(\eta, \partial D)\right),
	\end{equation}
	where \(\Lambda^*(\eta, \partial D)\) is the total mass of loops on $\C$ intersecting both $\eta$ and $\partial D$ under the normalized Brownian loop measure defined in~\cite{field2013reversed}. Then for any two conformally equivalent domains \(D\) and \(D'\), the pushforward of \(\mu_D\) under any conformal map from \(D\) to \(D'\) equals \(\mu_{D'}\). 
\end{definition}

The study of loop versions of SLEs can be traced back to Werner~\cite{werner2003sles}. Following this, Kontsevich and Suhov wondered whether the restriction property of these loops could be alternatively expressed in the framework of CFT, which led to the formalism described in \cite{kontsevich2007malliavin}. They also conjectured the existence and uniqueness of the measure in Definition~\ref{def:MKS} for all $c\le 1$ \cite[Conjecture~1]{kontsevich2007malliavin}, and related it to the unitarizing measure considered by Airault and Malliavin~\cite{airault2001unitarizing}. Now this measure is often known as the Malliavin-Kontsevich-Suhov measure.

Such an MKS measure was first considered by Werner~\cite{werner2008conformally} for $c=0$, where it is proved that the measure $\mathcal{W}$ induced from $\mu^{\text{BL}}$ by taking outer boundaries of Brownian loops indeed characterizes the $c=0$ MKS measure.
We will call $\mathcal{W}$ Werner's $\SLE_{8/3}$ loop measure on $\C$.
A construction for the case $c=-2$ was given in~\cite{benoist2016sle}.
Later, Kemppainen and Werner~\cite{kemppainen2016nested} constructed MKS measures for \( c \in (0,1] \) by taking the counting measure on loops in the full-plane \(\CLE_\kappa\) configuration, denoted by \(\SLE_\kappa^\lp\) in the sequel. In particular, they studied the $\SLE_{8/3}$ loop soup on $\C$ and proved that the outer and inner boundaries of $\SLE_{8/3}$ loop soup clusters are both equal to $\SLE_\kappa^\lp$. The construction for all \( c \leq 1 \) was finally completed by Zhan~\cite{zhan2021sle}.
Recently, Baverez and Jego~\cite{baverez2024cft} rigorously derived the CFT of SLE loop measures, based on an in-depth study of its Virasoro algebra structures. As a consequence, they obtained the uniqueness of MKS measures for all \( c \leq 1 \). See also a contemporaneous work~\cite{gordina2025infinitesimal} on the Virasoro representation of SLE loops.

The second main result of this paper is a direct and more ``probabilistic'' proof in the regime \( c \in (0,1] \), following a similar approach used in the generalized chordal restriction measures.

\begin{theorem}\label{thm:loop-unique}
	Let $c\in(0,1]$. The MKS measure is unique up to a positive multiplicative constant.
\end{theorem}

Similar to Theorem~\ref{thm:chordal-unique}, we will prove Theorem~\ref{thm:loop-unique} by relating the MKS loop measure to loop-soup observables. However, since the Brownian loop soup on $\C$ has only one cluster, one needs to consider the $\SLE_{8/3}$ loop soup instead, as indicated in~\cite{kemppainen2016nested}. This approach is feasible thanks to~\cite{carfagnini2024onsager}; namely, we can study a related measure $\wh\mu_D$, which is defined in a similar fashion to \eqref{eq:mks} but with $\Lambda^*$ replaced by Werner's $\text{SLE}_{8/3}$ loop measure $\cW$ on $\C$, see \eqref{eq:hmu}.

Let $\mathcal{L}_\C^{\mathcal{W}}$ be the whole-plane $\text{SLE}_{8/3}$ loop soup of intensity $\frac{c}{2}$, that is, a Poisson point process with intensity measure $\frac{c}{2}\mathcal{W}$.
Suppose $D$ and $U$ are simply connected domains in $\C$ containing the origin, and $U \subset D$. Let $\mathcal{C}_{\partial D}$ be the closure of the union of $\partial D$ and all clusters in $\mathcal{L}_\C^{\mathcal{W}}$ that intersect $\partial D$. 
By~\cite[Section 2.1]{kemppainen2016nested} and inversion invariance of $\mathcal{L}^\cW_\C$, the origin a.s. has a positive distance from $\mathcal{C}_{\partial D}$.
Hence, we can define properly the (open) connected component of $\C \setminus \mathcal{C}_{\partial D}$ that contains the origin by $D_\infty$. Furthermore, let $f_{D,\infty}$ be the conformal maps from $D_\infty \to \D$ such that $f_{D,\infty}(0)=0$ and $f_{D,\infty}'(0)>0$. Define $f_{U,\infty}$ in the same way. 
Then we have the following analog of Theorem~\ref{thm:chordal-intersect}.

\begin{theorem}\label{thm:intersect}
	Let $c\in(0,1]$. Let $\mu_\C^0$ be the restriction of $\mu_\C$ to loops surrounding the origin. Then there is a constant $\lambda\in(0,\infty)$ such that for all $U \subset D$ as above,
	\[
	\mu_\C^0[\ell\not\subset U,\ell\subset D]=\lambda\bE \log f_{U,\infty}'(0)-\lambda\bE\log f_{D,\infty}'(0),
	\]
	where $\bE$ denotes the expectation with respect to the $\SLE_{8/3}$ loop soup $\mathcal{L}_\C^{\mathcal{W}}$.
\end{theorem}

We emphasize that $\lambda$ depend on $c$ and the specific choice of $\mu_\C$.
Theorem~\ref{thm:intersect} is new and it directly implies Theorem~\ref{thm:loop-unique}. Note that for $c = 0$ and $\mu_\C = \mathcal{W}$, the corresponding result of Theorem~\ref{thm:intersect} has been provided in~\cite[Proposition 3]{werner2008conformally} with $\lambda=\frac{\pi}{5}$~\cite[Page 151]{werner2008conformally}.
Furthermore, if one views the sample of $\mu_\C$ as the outer boundary of the union of the loop from $\mathcal{W}$ and the clusters in $\mathcal{L}_\C^\mathcal{W}$ intersecting it, then Theorem~\ref{thm:intersect} is a direct consequence of its $c=0$ counterpart~\cite[Proposition 3]{werner2008conformally}.
Moreover, if $c\in(0,1]$ and $\kappa=\kappa(c)\in (\frac{8}{3},4]$, with $\mu_\C$ chosen as $\SLE_\kappa^\lp$ (i.e., the counting measure on loops of the full-plane $\text{CLE}_\kappa$), then the corresponding $\lambda$ in Theorem~\ref{thm:intersect} can be explicitly determined as $\lambda = \frac{1}{\pi}(\frac{\kappa}{4} - 1)\cot(\pi(1 - \frac{4}{\kappa}))$; 
see Corollary~\ref{cor:lambda0}.

\subsection{Applications}\label{sec:implication}

Here we provide several consequences of the results in Sections~\ref{sec:intro-chordal} and~\ref{sec:intro-loop}.

As mentioned to us by W. Werner, the \emph{reversibility} and \emph{duality} of SLE can be obtained from the uniqueness of these restriction measures, which was indeed considered (see e.g.~\cite{werner2005conformal}) before the proofs~\cite{zhan2008reversibility} (see also~\cite{miller2016imaginary2,lawler2021new})\footnote{We mention that the reversibility for non-simple SLE is established in~\cite{miller2016imaginary3}.} and~\cite{zhan2008duality,dubedat2009duality,miller2016imaginary1}. This approach to the reversibility of simple SLE loop measures was completed in~\cite[Section~7.1]{baverez2024cft}, as well as the $\kappa\leftrightarrow\frac{16}{\kappa}$ duality of $\SLE_\kappa$ loops~\cite[Theorem 7.1]{baverez2024cft} (this duality was also proved in a contemporaneous work~\cite[Theorem 1.2]{ang2024sle}). Using similar ideas, our results on the uniqueness of general chordal restriction measures (Theorem~\ref{thm:chordal-unique}) will provide alternate proofs in the regime $\kappa\in (\frac{8}{3},4]$ for the reversibility of chordal $\SLE_\kappa$ (and its variants), and the $\kappa\leftrightarrow\frac{16}{\kappa}$ duality in the chordal case.

The next application is about the SLE characterization of one-sided chordal restriction measures. 
Let $\gamma$ be the lower boundary of an one-sided chordal $0$-restriction sample with exponent $\alpha>0$, which is a simple path from $-1$ to $1$ on $\D$. Let $\Gamma$ be an independent $\CLE_\kappa$ on $\D$ with $\kappa\in(\frac{8}{3},4]$. Let $\eta$ be the lower boundary of the union of $\gamma$ and all loops of $\Gamma$ that $\gamma$ intersects. Then the law of $\eta$, which was characterized by $\SLE_\kappa(\rho)$ process in \cite{werner2013cle}, now can be deduced from Theorem~\ref{thm:chordal-unique} directly.
\begin{proposition}[{\cite[Theorem 1.1, Theorem 2.1]{werner2013cle}}]\label{prop:werner-wu}
	Let $\kappa\in(\frac{8}{3},4]$ and $\alpha>0$. 
	The law of $\eta$ is an $\SLE_\kappa(\rho)$ on $\D$ from $-1$ to $1$ with force point $e^{i\pi-}$, where $\rho\in(-2,\infty)$ satisfies $\alpha=\frac{(\rho+2)(\rho+6-\kappa)}{4\kappa}$.
\end{proposition}
\begin{proof}
	Since the $\CLE_\kappa$ on $\D$ can be realized as the outer boundaries of outermost clusters in a Brownian loop soup on $\D$ of intensity $\frac{c}{2}$~\cite{sheffield2012conformal}, it is straightforward to check that the connected component of $\D\setminus\eta$ above $\eta$ satisfies the one-sided chordal $c$-restriction with exponent $\alpha$ as in Definition~\ref{def:chordal}. On the other hand, as indicated in~\cite{MR2118865}, the upper part of $\SLE_\kappa(\rho)$ on $\D$
	satisfies the one-sided chordal $c$-restriction~\eqref{eq:chordal-res} with exponent $\alpha$ as well (see~\cite[Page 10-11]{werner2013cle} for a detailed justification). Combined, the result follows from the uniqueness part of Theorem~\ref{thm:chordal-unique}.
\end{proof}

Analogously, the two-sided case of Theorem~\ref{thm:chordal-unique} immediately gives the following. A radial counterpart will be given later by Proposition~\ref{prop:radial-eqv}.

\begin{proposition}\label{prop:chordal-eqv}
	Let $\kappa\in(\frac{8}{3},4]$ and $\alpha\ge\frac{6-\kappa}{2\kappa}$. Let $K$ be a two-sided chordal $c$-restriction with exponent $\alpha$. Let $\Gamma$ be an independent $\CLE_\kappa$ on $\D$. Let $\mathcal{C}(K)$ be the union of $K$ and all loops of $\Gamma$ that $K$ intersects. Then the filling of $\mathcal{C}(K)$ has the same law as that constructed in \cite[Proposition~6.2]{qian2021generalized} using variants of $\SLE_\kappa$.
\end{proposition}

Based on Theorem~\ref{thm:intersect}, for any connected compact set $V$ separating $0$ and $\infty$, we can express the $\mu_\C^0$-mass of loops intersecting $V$ by some $\SLE_{8/3}$ loop-soup observables; see \eqref{eq:et} below. As before, let $\mathcal{C}_V$ be the union of $V$ and all clusters in $\mathcal{L}_\C^{\mathcal{W}}$ that intersect $V$. Let $\Omega_{V,\infty}$ and $\Omega^*_{V,\infty}$ be the connected components of $\C\setminus \mathcal{C}_V$ containing $0$ and $\infty$, respectively. Let $V_\infty:=\C\setminus(\Omega_{V,\infty}\cup\Omega^*_{V,\infty})$, $f_{V_\infty}:\Omega_{V,\infty}\to\D$ be conformal fixing $0$, and $h_{V_\infty}:\Omega_{V,\infty}^*\to \D^*:=\C\setminus\overline\D$ be conformal fixing $\infty$.

\begin{theorem}\label{thm:electrical-thickness}
For any $\mu_\C$ with parameter $c\in(0,1]$, there is a constant $\zeta\in\R$ such that for any connected compact set $V$ separating $0$ and $\infty$,
	\begin{equation}\label{eq:et}
    \mu_\C^0[\ell\cap V\neq\emptyset]=\lambda\bE\left[\theta(V_\infty)\right]+\zeta,
    \end{equation}
	where $\theta(V_\infty):=\log|f_{V_\infty}'(0)|-\log|h_{V_\infty}'(\infty)|$ is the so-called electrical thickness of $V_\infty$~\cite{kw04}, and $\bE$ is the expectation for the $\SLE_{8/3}$ loop soup $\mathcal{L}_\C^{\mathcal{W}}$. In fact, $\zeta$ can be related to  $\lambda$ through \eqref{eq:zeta}.
\end{theorem}
The proof of Theorem~\ref{thm:electrical-thickness} will be given in Section~\ref{sec:ele-thik}. Theorem~\ref{thm:electrical-thickness} can also be straightforwardly extended to $c=0$ and $\mu_\C=\mathcal{W}$, in which case $\zeta=\frac{\sqrt{6}}{15}$; see Remark~\ref{rmk:c0}. However, for general $c\in(0,1]$ and $\mu_\C=\SLE_\kappa^\lp$, the exact value of $\zeta$ remains unknown.

\medskip
\textbf{Organization of the paper.}
In Section~\ref{sec:general}, we present the proof of Theorems~\ref{thm:chordal-unique} and~\ref{thm:chordal-intersect}, and discuss its radial counterpart. In Section~\ref{sec:loop}, we prove Theorems~\ref{thm:loop-unique} and~\ref{thm:intersect}. Then we show Theorem~\ref{thm:electrical-thickness} in Section~\ref{sec:ele-thik}. Finally in Section~\ref{sec:further}, we give some further remarks and discussions.

\section{Generalized chordal and radial restriction measures}\label{sec:general}

\subsection{The chordal case}\label{subsec:chordal}
In this section, we prove Theorem~\ref{thm:chordal-intersect} by using a Markovian exploration process (Definition~\ref{def:ep}) and a recursive formula (Proposition~\ref{prop:recursion}) derived from the restriction property \eqref{eq:chordal-res}. Based on these results, we then finish the proof Theorem~\ref{thm:chordal-unique}.

We first focus on the one-sided case. Let $c\in(0,1]$, $\alpha\in \R$, and $A\in\mathcal{Q}_-$.
Let $\bP$ and $\bE$ be the law and the expectation with respect to the Brownian loop soup $\cL_\D$ on $\D$ (of intensity $\frac{c}{2}$ as assumed). 
Now, we define a natural exploration process associated with a set $A\in\mathcal{Q}_-$ by revealing the loops connected to $A$ in $\cL_\D$ in a Markovian way.
\begin{definition}[Exploration process]\label{def:ep}
    Suppose $A\in\mathcal{Q}_-$. Let $A_0:=A$, and for all $n\ge 0$, let $A_{n+1}$ be the filling of the union of $A_n$ and all loops in $\cL_\D$ that intersect $A_n$.
\end{definition}

Recall that the filling of a set is closed, and hence each $A_n$ is closed. 
As introduced, $\wt A$ denotes the filling of the union of $A$ and all clusters in $\mathcal{L}_\D$ that intersect $A$.
Note that $A_n\cap\S^1=\wt A\cap\S^1=A\cap\S^1$ for all $n\ge0$ (see the discussion below \eqref{eq:E} and note that $A\subset A_n\subset \wt A$).
Recall that $w_0$ is some fixed point on $\S^1\setminus A$ other than $-1,1$, and hence $w_0$ also lies in $\S^1\setminus A_n=\S^1\setminus \wt A$. Let $f_{A_n,w_0}$ be the unique conformal map from $\D\setminus A_n$ to $\D$ that fixes $-1,1,w_0$. 
For sets $B_1, B_2\subset\C$, the Hausdorff distance between $B_1$ and $B_2$ is given by
\[
d_H(B_1,B_2):=\inf\{\eps>0:B_1\subset B_2^{(\eps)}, B_2\subset B_1^{(\eps)}\}, \text{ where } B_i^{(\eps)}:=\{z: \dist(z,B_i)<\eps\}.
\]
We first show that the exploration process $A_n$ converges to $\wt A$ in Hausdorff distance. 

\begin{lemma}\label{lem:dH}
     For any $A\in\mathcal{Q}_-$, we have $\bP$-a.s. $\lim_{n\to\infty}d_H(A_n,\wt A)=0$, and $f_{A_n,w_0}^{-1}\to f_{\wt A,w_0}^{-1}$ under the local uniform topology. 
\end{lemma}
\begin{proof}
   For any $\eps>0$, by \cite[Lemma 9.7]{sheffield2012conformal}, there are only finitely many clusters in $\cL_\D$  with diameter larger than $\eps$. We denote those that intersect $A$ by $\cC_1,\ldots,\cC_m$, which are at positive distance from each other. Below, we use $\Fill$ to denote the filling of a set for brevity. Then, $A':=\Fill(A\cup \cC_1\cup\ldots\cup\cC_m)\subset \wt A$, and $d_H(A',\wt A)\le \eps$.
    According to \cite[Theorem 4.1]{van2016random}, every cluster $\cC_i$, $1\le i\le m$, can be approximated by a sequence of finite subclusters $\cC_i^N$ of $\cC_i$ (i.e., a chain of finite number of loops in $\cC_i$) such that $\lim_{N\to\infty}d_H(\Fill(\cC_i^N),\Fill(\cC_i))=0$. Since each finite subcluster $\cC_i^N$ is contained in $A_n$ for sufficiently large $n$, we conclude $\lim_{n\to\infty}d_H(A_n,\wt A)=0$. Similarly, \cite[Theorem 4.1]{van2016random} also implies $\lim_{n\to\infty}d_H(\D\setminus A_n,\D\setminus\wt A)=0$. Since $\{f_{A_n,w_0}^{-1}\}_{n\ge1}$ is the sequence of conformal maps from $\D$ to $\D\setminus A_n$ fixing $\pm1$ and $w_0$, by a standard extension of Carath\'eodory kernel theorem (see Proposition~\ref{prop:bdy-caratheodory} for details), we conclude $f_{A_n,w_0}^{-1}\to f_{\wt A,w_0}^{-1}$ locally uniformly.
\end{proof}

We derive the following recursive formula for the non-intersection probability of restriction samples by the loop-soup interpolation.
\begin{proposition}\label{prop:recursion}
Let $c\in(0,1]$ and $\alpha\in \R$.
Suppose $\P$ satisfies the one-sided chordal $c$-restriction with exponent $\alpha$ as in Definition~\ref{def:chordal}.
For any $A\in\mathcal{Q}_-$, we have
\begin{equation*}
    \P[K\cap A=\emptyset]=\bE\otimes\E\!\left[(f_{A_n,w_0}'\!(1)f_{A_n,w_0}'\!(-1))^\alpha \exp\!\Big(\!\!-\!\frac c2 \Lambda_{\D}\big((f_{A_{n-1},w_0}\!\!\circ\! f_{A_n,w_0}^{-1})(K),f_{A_{n-1},w_0}\!(A_n\!\setminus\! A_{n-1})\big)\!\Big)\!\right].
\end{equation*}
\end{proposition}
\begin{proof}
Given $K\cap A=\emptyset$, there exists $\eps>0$ (random w.r.t. $\P$) such that $K^{(2\eps)}\cap A=\emptyset$ since both $K$ and $A$ are closed. Then, the number of Brownian loops in $\mathcal{L}_\D$ that intersect both $K^{(\eps)}$ and $A$ is a.s. finite since these loops have diameter at least $\eps$\footnote{Here we use the basic fact that the loops in the Brownian loop soup are locally finite, i.e., for any $\eps>0$, there are only finitely many loops in $\mathcal{L}_\D$ that have diameter greater than $\eps$
(equivalently, the $\mu^{\rm BL}$-mass of loops in $\D$ that have diameter greater than $\eps$ is finite). Indeed, repeating the similar computations as in the proof of \cite[Lemma~2.13]{camia2017scaling} (setting the mass function $m\equiv 0$ therein), we see that the total mass is bounded by $\int_\eps^2 \frac{\pi}{4}\frac{(2+r)^2}{2\pi r^2} dr + \int_\eps^2 \frac{\pi}{4} (2+r)^2 \int_0^r \frac{2}{\pi t^2} e^{-r^2/(288t)}\, dt\, dr<\infty$.}. Hence, almost surely, $K\cap A_1=\emptyset$ iff $K\cap A_1'=\emptyset$, where $A_1'$ is the union of $A$ and all the loops in $\cL_\D$ that intersect $A$, and $A_1$ is the filling of $A_1'$. Therefore, by the property of Poisson Point process, given $K$ such that $K\cap A=\emptyset$, we have
\begin{equation}\label{eq:KA1}
    \bP[K\cap A_1=\emptyset]=\bP[K\cap A'_1=\emptyset]=\exp(-\frac c2 \Lambda_{\D}(K,A)).
\end{equation}
Now, according to~\eqref{eq:chordal-res}, it holds that
\begin{align}
    \P[K\cap A=\emptyset]&=\E_{ A}\big[(f_{A,w_0}'(1)f_{A,w_0}'(-1))^\alpha\exp\big(-\frac c2 \Lambda_{\D}(K,A)\big)\big]\notag\\
    &=\bE\otimes\E_{A}\left[(f_{A,w_0}'(1)f_{A,w_0}'(-1))^\alpha {\bf 1}_{K\cap A_1=\emptyset}\right]\label{eq:recursion-2}\\ 
    &=\bE\otimes\E_{A_1}\big[(f_{A_1,w_0}'(1)f_{A_1,w_0}'(-1))^\alpha \exp\big(-\frac c2 \Lambda_{\D\backslash A}(K,A_1)\big)\big],\label{eq:recursion-1}
\end{align}
where we used \eqref{eq:KA1}
in~\eqref{eq:recursion-2}, and used the following equation  in~\eqref{eq:recursion-1} 
$$\frac{d\P_{A}(K)}{d\P_{A_1}(K)}{\bf 1}_{K\cap A_1=\emptyset}={\bf 1}_{K\cap A_1=\emptyset}(\phi'(1)\phi'(-1))^\alpha\exp(-\frac c2 \Lambda_{\D\backslash A}(K,A_1)) \ \text{ for } \phi=f_{A_1,w_0}\circ f_{A,w_0}^{-1}.$$ 
By induction, we have for any $n\ge1$, 
\begin{equation*}
    \P[K\cap A=\emptyset]=\bE\otimes\E_{ A_n}\left[(f_{A_n,w_0}'(1)f_{A_n,w_0}'(-1))^\alpha \exp\left(-\frac c2 \Lambda_{\D\backslash A_{n-1}}(K,A_n)\right)\right].
\end{equation*}
Now, by the conformal invariance of Brownian loop measure, 
\[
\Lambda_{\D\backslash A_{n-1}}(K,A_n)=\Lambda_{\D}\big(f_{A_{n-1},w_0}(K),f_{A_{n-1},w_0}(A_n\setminus A_{n-1})\big).
\]
Moreover, if $K$ is sampled from $\P$, then $f_{A_n,w_0}^{-1}(K)$ is distributed as $\P_{A_n}$. The lemma then follows.
\end{proof}

Next, we give a simple lemma on the Brownian loop soup.
\begin{lemma}\label{lem:reaching}
    Let $0<a<\frac{1}{2}$. Suppose $B$ is a connected subset of $\D$ such that $\partial B\cap\S^1\neq\emptyset$ and $B\cap(1-a)\S^1\neq\emptyset$. Let $H$ be the event that there is a loop in $\mathcal{L}_\D$ that intersects both $B$ and $\frac12\D$. Then, there exists a constant $\delta=\delta(c,a)>0$ such that for all $B$,
    \[
    \bP[H]\ge \delta.
    \]
\end{lemma}
\begin{proof}
    Let $z$ be some point on $B\cap(1-a)\S^1$. Let $L$ be the line segment connecting $z$ and $z/(2|z|)$. For any $b>0$, define $L(b):=\{w: \dist(w,L)\le b\}$. Let $\Lambda(z)$ be the set of loops in $L(\frac{a}{10})\setminus L(\frac{a}{20})$ that surround $L(\frac{a}{20})$. Note that $B$ contains a curve intersecting both $\partial L(\frac{a}{10})$ and $\partial L(\frac{a}{20})$. Every loop in $\Lambda(z)$ will intersect $B$. Hence, $\bP[H]\ge\bP[\mathcal{L}_\D\cap\Lambda(z)\neq\emptyset]$.
    By the conformal invariance of $\mu^{\rm BL}$ and the fact that the total mass of loops in any annulus with nonzero winding number is positive~\cite[Proposition 3.9]{lawler2011defining}, we get that $\mu^{\rm BL}[\Lambda(z)]\ge \bar\delta(a)$ for some constant $\bar\delta(a)>0$ independent of $z$. This allows us to conclude the proof.
\end{proof}

The following lemma shows that $f_{A_{n-1},w_0}(A_n\setminus A_{n-1})$ will stay close to the boundary of $\D$ as $n\to\infty$.
\begin{lemma}\label{lem:distance}
For any $A\in\mathcal{Q}_-$, let $d_n:=\sup\{{\rm dist}(z,\S^1)\mid z\in  f_{A_{n-1},w_0}(A_n\setminus A_{n-1})\}$.
Under $\bP$, $d_n\to0$ in probability.
\end{lemma}
\begin{proof}
Otherwise, there is $\varepsilon_0>0$, $\rho>0$ and a sequence $n_i\to\infty$ such that $\bP[d_{n_i}>\varepsilon_0]>\rho$. Then by Lemma~\ref{lem:reaching}, given $d_{n_i}>\varepsilon_0$, the probability that there is a loop in $\cL_\D$ intersecting both $f_{A_{n_i-1},w_0}(A_{n_i}\setminus A_{n_i-1})$ and $\frac 12\D$ 
is bounded from below by a constant $\delta(\varepsilon_0)>0$ only depending on $\varepsilon_0$. Hence, by conformal invariance of the Brownian loop soup,
\begin{equation}\label{eq:prob-markov}
    \bP\Big[f_{A_{n_i-1},w_0}(A_{n_i+1}\setminus A_{n_i})\cap\Big(\frac{1}{2}\D\Big)\neq\emptyset\Big]>\rho\delta(\varepsilon_0).
\end{equation}
By Lemma~\ref{lem:dH},
$f_{A_n,w_0}^{-1}\to f_{\wt A,w_0}^{-1}$ locally uniformly a.s.  Then for any $\sigma>0$ and sufficiently large $n$, 
\begin{equation}
\bP\left[f_{A_n,w_0}^{-1}\left(\frac{1}{2}\D\right)\not\subset f_{\wt A,w_0}^{-1}\left(\frac{1+\sigma}{2}\D\right)\right]<\frac{1}{2}\rho\delta(\varepsilon_0).
\end{equation}
Combining with~\eqref{eq:prob-markov}, we obtain that
\(
    \bP\left[\left(A_{n_i+1}\setminus A_{n_i}\right)\cap f_{\wt A,w_0}^{-1}\left(\frac{1+\sigma}{2}\D\right)\neq\emptyset\right]>\frac{1}{2}\rho\delta(\varepsilon_0)
\)
for sufficiently large $n_i$, which contradicts with that $A_{n_i+1}\subset \wt A$ and $\wt A\cap f_{\wt A,w_0}^{-1}\left(\frac{1+\sigma}{2}\D\right)=\emptyset$.
\end{proof}

The following lemma deals with the behavior of $(f_{A_n,w_0})_{n\ge1}$ near the boundary points $-1,1$. Recall that $f_{A_n,w_0}$ is normalized such that it preserves $\pm 1$ and $w_0$ (see the paragraph below Definition~\ref{def:ep}).
\begin{lemma}\label{lem:boundary}
$\bP$-almost surely, for any $A\in\mathcal{Q}_-$, all of the followings hold. First, $f_{A_n,w_0}'(\pm 1)\downarrow f_{\wt A,w_0}'(\pm 1)$. Next, the distance between $f_{A_{n-1},w_0}(A_n\setminus A_{n-1})$ and $\{-1,1\}$ has a (random) uniform positive lower bound. Furthermore, for any $K\in\mathcal{K}$, $(f_{A_{n-1},w_0}\circ f_{A_n,w_0}^{-1})(K)\to K$ in Hausdorff distance.
\end{lemma}
\begin{proof}
Let $B_r(1):=\{z:|z-1|<r\}$ for any $r>0$. Let $\varepsilon<\frac{1}{100}$ be a random positive
number such that $B_{\varepsilon}(1)\cap \wt A=\emptyset$ and $B_{2\varepsilon}(1)\cap f_{\wt A,w_0}^{-1}(\frac{1}{2}\D)=\emptyset$. 
Then, since $A_n\subset\wt A$ and a.s. $f_{A_n,w_0}^{-1}\to f_{\wt A,w_0}^{-1}$ locally uniformly by Lemma~\ref{lem:dH}, there is a random $N$ such that $B_{\varepsilon}(1)\cap A_n=\emptyset$ and $B_{2\varepsilon}(1)\cap f_{A_n,w_0}^{-1}(\frac{1}{2}\D)=\emptyset$ for all $n\ge N$.
By Schwarz reflection for circular arcs, $f_{A_n,w_0}$(for $n\ge N$) and $f_{\wt A,w_0}$ can be analytically extended to functions $\varphi_n$ and $\varphi$ defined on $B_\varepsilon(1)\cup(\D\setminus\wt A)$, respectively. Note that for $n\ge N$, $|\varphi_n(z)|\le 2$ for any $z\in B_\varepsilon(1)\cup(\D\setminus\wt A)$. Therefore, $\{\varphi_n\}$ forms a normal family. Moreover, every limit in this family, when restricted to $\D\setminus\wt A$, is identical to $\varphi\big|_{\D\setminus\wt A}=f_{\wt A,w_0}$ by Lemma~\ref{lem:dH}. By the uniqueness of analytic function, we then conclude $\varphi_n\to\varphi$ locally uniformly, and as an interior point in $B_\varepsilon(1)\cup(\D\setminus\wt A)$, $f_{A_n,w_0}'(1)=\varphi_n'(1)$ converges to $f_{\wt A,w_0}'(1)=\varphi'(1)$. 
In particular, $|\varphi_n'(1)|\ge C$ for some random $C>0$ and hence $\varphi_n(B_\varepsilon(1))\supset B_{C\varepsilon/4}(1)$ by Koebe's 1/4 theorem. Hence $f_{A_{n-1},w_0}(A_n\setminus A_{n-1})$, as a subset of $f_{A_{n-1},w_0}(\wt A\setminus A_{n-1})$, has distance $\ge\frac{1}{4}C\varepsilon$ from $1$. Similar results hold for the case $-1$. The last claim follows directly from the local uniform convergence of $f_{A_n,w_0}$ and $\varphi_n$.
\end{proof}

Finally, we are ready to show Theorem~\ref{thm:chordal-intersect}.
\begin{proof}[Proof of Theorem~\ref{thm:chordal-intersect}]
    We first finish the proof for the one-sided case, which requires the use of Proposition~\ref{prop:recursion} by taking $n\to\infty$.
    By Lemmas~\ref{lem:distance} and~\ref{lem:boundary}, we conclude that under $\bP\otimes\P$,
    \[ 
    \Lambda_{\D}\big((f_{A_{n-1},w_0}\circ f_{A_n,w_0}^{-1})(K),f_{A_{n-1},w_0}(A_n\setminus A_{n-1})\big)\to 0 \text{ in probability.} 
    \]
    Note that $\bP$-a.s. $f_{A_n,w_0}'(\pm 1)\in (0,1)$ and $f_{A_n,w_0}'(\pm 1)\downarrow f_{\wt A,w_0}'(\pm 1)$ by Lemma~\ref{lem:boundary}. Therefore,
    \begin{gather}\begin{split}\label{eq:convg}
&\bE\otimes\E\left[(f_{A_n,w_0}'(1)f_{A_n,w_0}'(-1))^\alpha \exp\Big(-\frac c2 \Lambda_{\D}\big((f_{A_{n-1},w_0}\circ f_{A_n,w_0}^{-1})(K),f_{A_{n-1},w_0}(A_n\setminus A_{n-1})\big)\Big)\right]\\[2mm]
        &\quad \quad \to\bE\left[\left(f_{\wt A,w_0}'(1)f_{\wt A,w_0}’(-1)\right)^\alpha\right] \quad \quad  \text{ as } n\to\infty, \end{split}
    \end{gather}
    where one can use the monotone convergence theorem to show that $\limsup$ of LHS is smaller than RHS, and use Fatou's lemma to show that $\liminf$ of LHS is greater than RHS.  
    Hence, by Proposition~\ref{prop:recursion}, we conclude $\P[K\cap A=\emptyset]=\bE[(f_{\wt A,w_0}'(1)f_{\wt A,w_0}’(-1))^\alpha]$ and finish the one-sided case. 
    
    The two-sided case is proved in a similar way. However, compared to the one-sided case, now it could happen that $-1$ and $1$ are separated by $\wt A$ inside $\D$ (e.g.\ $A$ contains two disjoint islands attached to the lower and upper parts of $\S^1$ resp. and they are connected by a loop-soup cluster to separate $-1$ and $1$). 
    Below, we only give the necessary modifications and omit further details.
    Let $\P$ be the two-sided chordal $c$-restriction measure with exponent $\alpha$ (Definition~\ref{def:chordal}).
    For $A\in\mathcal{Q}$,
    the exploration process $(A_n)_{n\ge 0}$ in Definition~\ref{def:ep} is still valid, and $\bP$-a.s. $\lim_{n\to\infty}d_H(A_n,\wt A)=0$ as well as $\lim_{n\to\infty}d_H(\D\setminus A_n,\D\setminus \wt A)=0$. Furthermore, if we let $E_n$ be the event that $-1$ and $1$ are in the same connected component of $\overline{\D}\setminus A_n$, then $\bP$-a.s. ${\bf 1}_{E_n}\downarrow{\bf 1}_{E}$, where $E$ is defined in \eqref{eq:E}.
    Similarly, if we let $f_{A_n,w_0}$ be the conformal map from $\D\setminus A_n$ to $\D$ fixing $\pm1$ and $w_0$, then $\bP$-a.s. on $E$, $f_{A_n,w_0}^{-1}\to f_{\wt A,w_0}^{-1}$ locally uniformly (using Proposition~\ref{prop:bdy-caratheodory} again). 
    The same argument as in the proof of Proposition~\ref{prop:recursion} now gives
    \begin{align*}
    &\P[K\cap A=\emptyset]\\
    =\ &\bE\otimes\E\left[{\bf 1}_{E_n}(f_{A_n,w_0}'(1)f_{A_n,w_0}'(-1))^\alpha \exp\Big(-\frac c2 \Lambda_{\D}\big((f_{A_{n-1},w_0}\circ f_{A_n,w_0}^{-1})(K),f_{A_{n-1},w_0}(A_n\setminus A_{n-1})\big)\Big)\right].
\end{align*}
For $A\in\mathcal{Q}$ and on the event $E$, analogues of Lemmas~\ref{lem:distance} and~\ref{lem:boundary} hold as well. Then similar to \eqref{eq:convg}, the above converges to $\bE[{\bf 1}_{E}(f_{\wt A,w_0}'(1)f_{\wt A,w_0}’(-1))^\alpha]$, which implies \eqref{eq:general}, as desired.
\end{proof}

\begin{proof}[Proof of Theorem~\ref{thm:chordal-unique}]
Let $c\in(0,1]$. The existence of the two-sided chordal $c$-restriction with exponent $\alpha$ for $\alpha\ge\frac{6-\kappa}{2\kappa}$ is given by~\cite[Proposition 6.2]{qian2021generalized}. 
Moreover, according to~\cite{dubedat2009duality}, the chordal $\SLE_\kappa(\rho)$ with $\rho\in(-2,\infty)$ is a one-sided chordal $c$-restriction sample with exponent $\alpha$ such that $\alpha=\frac{(\rho+2)(\rho+6-\kappa)}{4\kappa}$, giving the existence for all $\alpha>0$.

The uniqueness part of Theorem~\ref{thm:chordal-unique} follows immediately from Theorem~\ref{thm:chordal-intersect}. It remains to show that the two-sided (resp.\ one-sided) chordal $c$-restriction measure with exponent $\alpha$ exists if and only if $\alpha\ge\frac{6-\kappa}{2\kappa}$ (resp.\ $\alpha>0$).
For the one-sided case, when $\alpha\le0$, the right side of~\eqref{eq:general} is always $\ge1$, which cannot be the case.
The argument for the two-sided case is similar to \cite[Corollary~8.6]{lawler2003conformal}.
 Suppose that there is a probability measure $\P$ on random sets $K\in\mathcal{K}$ satisfying the two-sided chordal $c$-restriction with exponent $\alpha$ for some $\alpha<\frac{6-\kappa}{2\kappa}$.
 Then the lower boundary $\gamma$ of $K$ satisfies the one-sided chordal $c$-restriction with the same exponent $\alpha$; hence $\alpha$ needs to be positive. According to the uniqueness of the one-sided case, $\gamma$ is then an $\SLE_\kappa(\rho)$ with $\alpha=\frac{(\rho+2)(\rho+6-\kappa)}{4\kappa}$. Since $\alpha<\frac{6-\kappa}{2\kappa}$, here $\rho<0$. In particular, $\P[0\text{ is above }\gamma]<\frac{1}{2}$. However, the symmetry of $\P$ implies $\P[0\text{ is above }\gamma]$ is at least $\frac{1}{2}$ (it can be strictly larger when $K$ is of positive Lebesgue measure), a contradiction.
\end{proof}

\subsection{The radial case}\label{sec:radial}
In this section, we provide the radial counterpart.
To lighten notation, we use the same symbols as the chordal case to denote their radial counterparts. Let $\cK$ be the collection of all simply connected compact sets $K\subset\overline{\D}$ such that $0\in K$ and $K\cap\S^1=\{1\}$. Let $\cQ$ be the collection of all compact sets $A\subset\overline{\D}$ such that $\D\setminus A$ is simply connected and $0,1\notin A$. For any $A\in \cQ$, let $f_A:\D\setminus A\to \D$ be the conformal map that fixes $0,1$. The following definition for general $c\le1$ is from \cite[Definition~1.3]{qian2021generalized}, which extends the $c=0$ case considered in~\cite{wu2015conformal}. 

\begin{definition}\label{def:radial}
    Let $c\le1$ and $\alpha,\beta\in\R$.
    We say that a probability measure $\P$ on $\mathcal{K}$ satisfies the radial $c$-restriction with exponents $(\alpha,\beta)$, if for all $A\in\mathcal{Q}$,
    \begin{equation}\label{eq:radial-res}
        \frac{d\P(K)}{d\P_A(K)}{\bf 1}_{K\cap A=\emptyset}={\bf 1}_{K\cap A=\emptyset}|f_A'(0)|^\alpha f_A'(1)^\beta \exp\Big(-\frac{c}{2}\Lambda_{\D}(K,A)\Big),
    \end{equation}
    where $\P_A=\P\circ f_A$ and $\Lambda_{\D}(K,A)$ is the same as in Definition~\ref{def:chordal}.
\end{definition}

In \cite[Theorem~1.6]{qian2021generalized}, the author used the radial hypergeometric $\SLE_\kappa$ to construct a probability measure $\P_\kappa^{\alpha,\beta}$, which satisfies the radial $c$-restriction with exponents $(\alpha,\beta)$, in the range 
\begin{equation}\label{eq:regime}
    \alpha\le\eta_\kappa(\beta),\  \beta\ge \frac{6-\kappa}{2\kappa}, \text{ where } \eta_\kappa(\beta):=\frac{(\sqrt{16\kappa\beta+(4-\kappa)^2}-(4-\kappa))^2-4(4-\kappa)^2}{32\kappa}
\end{equation}
is the generalized disconnection exponent~\cite[(1.7)]{qian2021generalized}. The author conjectured that \eqref{eq:regime} is the maximal range such that the measure in Definition~\ref{def:radial} exists, and further conjectured that for any $c\le 1$ ($\kappa\in (0,4]$), the restriction property \eqref{eq:radial-res} uniquely characterizes the law of $\P$. Up to now, the only proved case is $c=0$~\cite{wu2015conformal}.
Using the approach similar to the chordal case, we are able to confirm the uniqueness in the range \eqref{eq:regime}.


\begin{theorem}\label{thm:radial}
    Let $c\in(0,1]$ and $(\alpha,\beta)$ in the range \eqref{eq:regime}.
    Suppose $\P$ satisfies the radial $c$-restriction with exponents $(\alpha,\beta)$. Then, for all $A\in\mathcal{Q}$, we have
    \begin{equation}\label{eq:general-radial}
        \P[K\cap A=\emptyset]=\bE\left[{\bf 1}_{0\notin\wt A}\,|f_{\wt A}'(0)|^\alpha \,f_{\wt A}'(1)^\beta \right],
    \end{equation}
    where $\bE$ denotes the expectation for the Brownian loop soup $\mathcal{L}_\D$ of intensity $\frac{c}{2}$, $\wt A$ is the filling of the union of $A$ and all clusters in $\mathcal{L}_\D$ that intersect $A$, and $f_{\wt A}$ is the conformal map from $\D\setminus\wt A$ to $\D$ that fixes $0,1$. As a result, for the above $c,\alpha,\beta$, the probability measure on $\mathcal{K}$ that satisfies the radial $c$-restriction with exponents $(\alpha,\beta)$ as in Definition~\ref{def:radial} is unique.
\end{theorem}

\begin{proof}
    It suffices to show \eqref{eq:general-radial}.
    For any $A\in\mathcal{Q}$, the exploration process $(A_n)_{n\ge 0}$ given in Definition~\ref{def:ep} is still valid, and we have $\bP$-a.s. $\lim_{n\to\infty}d_H(A_n,\wt A)=0$.  
    Moreover, $\bP$-a.s. ${\bf 1}_{0\notin A_n}\downarrow{\bf 1}_{0\notin \wt A}$. Similar to Proposition~\ref{prop:recursion}, we can use \eqref{eq:radial-res} to deduce the following formula 
    \begin{gather}\begin{split}\label{eq:recursion-radial}
    &\P[K\cap A=\emptyset]\\
    =\ &\bE\otimes\E\left[{\bf 1}_{0\notin A_n}\,|f_{A_n}'(0)|^\alpha \,f_{A_n}'(1)^\beta \exp\Big(-\frac c2 \Lambda_{\D}\big((f_{A_{n-1}}\circ f_{A_n}^{-1})(K),f_{A_{n-1}}(A_n\setminus A_{n-1})\big)\Big)\right].
    \end{split}
\end{gather}
On the event $\{0\notin \wt A\}$, we have $f_{A_n}^{-1}\to f_{\wt A}^{-1}$, $1\le |f_{A_n}'(0)|\uparrow |f_{\wt A}'(0)|$, $1\ge f_{A_n}'(1)\downarrow f_{\wt A}'(1)>0$, and the exponential term tends to $1$ as before. 
By the bounded convergence theorem, we get \eqref{eq:general-radial} by letting $n\to\infty$ in \eqref{eq:recursion-radial}. 
More concretely, $|f_{A_n}'(0)|\le 4r_n^{-1}$ with $r_n:=\dist(0,A_n)$ by Koebe’s 1/4 theorem. Note that $f_{A_n}'(1)$ is equal to the probability that a Brownian excursion from $0$ to $1$ in $\D$ avoids $A_n$, and hence, $f_{A_n}'(1)\le Cr_n^{1/2}$ for some constant $C>1$ by the Beurling estimate. Combined, for any $\beta\ge0$ and $\alpha\le\beta/2$, the quantity inside $\bE\otimes\E$ in \eqref{eq:recursion-radial} is bounded by $4^\alpha C^\beta$, which allows us to conclude \eqref{eq:general-radial} in this range. Noting that $\beta\ge0$ and $\alpha\le\beta/2$ contains the range \eqref{eq:regime}, i.e. $\eta_\kappa(\beta)\le \frac\beta2$ and $\frac{6-\kappa}{2\kappa}>0$ for any $\kappa\in(0,4]$ and $\beta\ge0$, we complete the proof. 
\end{proof}

Consequently, we obtain the following result, which plays a crucial role in determining the Hausdorff dimension of certain exceptional sets arising from loop-soup clusters \cite{gao2022multiple}.

\begin{proposition}\label{prop:radial-eqv}
    Let $\kappa\in(\frac{8}{3},4]$, $\beta\ge\frac{5}{8}$ and $\alpha\le\eta_{8/3}(\beta)$, where $\eta_{8/3}(\beta)$ given by \eqref{eq:regime} is the Brownian disconnection exponent. Let $K$ be a standard radial restriction sample in $\D$ with exponents $(\alpha,\beta)$ and marked points $0$ and $1$ (i.e. the $c=0$ case in Definition~\ref{def:radial}). Let $\Gamma$ be an independent $\CLE_\kappa$ on $\D$. Let $\mathcal{C}(K)$ be the union of $K$ and all loops of $\Gamma$ that $K$ intersects. Then, the filling of $\mathcal{C}(K)$ has the same law as $\P_\kappa^{\alpha,\beta}$, which is constructed by the radial hypergeometric $\SLE_\kappa$ in \cite[Section 4.1]{qian2021generalized}.
\end{proposition}
\begin{proof}
    Note that both the law of the filling of $\mathcal{C}(K)$ and $\P_\kappa^{\alpha,\beta}$ satisfy the radial $c$-restriction with exponents $(\alpha,\beta)$; see Theorem~1.6 and Lemma~2.1 of \cite{qian2021generalized}, respectively. Now, Theorem~\ref{thm:radial} gives the equivalence between these two.
\end{proof}

\begin{remark}
    Our analysis can be also adapted to the trichordal case; see \cite{qian2018conformal} and \cite[Section~6.2.2]{qian2021generalized} for related notions. We leave it to the interested reader.
\end{remark}

\section{The Malliavin-Kontsevich-Suhov measure}\label{sec:loop}

This section focuses on the MKS loop measure $\mu_\C$ (see Definition~\ref{def:MKS}) on the whole-plane. We will prove Theorem~\ref{thm:intersect} in Section~\ref{subsec:hitting}, which implies Theorem~\ref{thm:loop-unique} immediately.
We remind the reader that all the constants, $\lambda$ and $\zeta_1$ in this section as well as $\zeta_2$ and $\zeta$ in the next section, will depend on $c$ and the choice of $\mu_\C$ (which is not unique).

As mentioned before, since the Brownian loop soup on $\C$ a.s. has only one cluster for any $c>0$, we will use the \emph{$\SLE_{8/3}$ loop soup} on $\C$ instead, as considered in~\cite{kemppainen2016nested}.
Recall that Werner's $\SLE_{8/3}$ loop measure $\mathcal{W}$ on $\C$ is induced from the Brownian loop measure $\mu^{\text{BL}}$ on $\C$ by taking outer boundaries. For an MKS loop measure $\mu_\C$ and $c\in(-\infty,1]$, on any simply connected $D\subset\C$, define $\wh \mu_D$ by
\begin{equation}\label{eq:hmu}
	\frac{d\wh \mu_D(\eta)}{d\mu_\C(\eta)}{\bf 1}_{\eta\subset D}={\bf 1}_{\eta\subset D}\exp\Big(\frac{c}{2}\mathcal{W}(\eta,\partial D)\Big),
\end{equation}
where $\mathcal{W}(\eta,\partial D)$ is the total mass of loops on $\C$ intersecting both $\eta$ and $\partial D$ under the measure $\mathcal{W}$.
We first show that the above $(\wh\mu_D)_{D\subset\C}$ is also conformally invariant as the family $(\mu_D)_{D\subset\C}$ defined in~\eqref{eq:mks}, based on a relation between $\mathcal{W}$ and $\mu^{\rm BL}$ obtained by Carfagnini and Wang~\cite{carfagnini2024onsager}.
\begin{lemma}\label{lem:ci-hmu}
	$(\wh\mu_D)_{D\subset\C}$ defined above is conformally invariant. That is, for $\phi:D\to D'$ conformal, the pushforward of $\wh\mu_D$ under $\phi$ equals $\wh\mu_{D'}$.
\end{lemma}
\begin{proof}
	Since $\mu_\C$ is an MKS measure, we have $\phi\circ\mu_D=\mu_{D'}$, and by \eqref{eq:mks},
	\begin{equation*}
		\int_{\{\eta\subset D\}}F(\phi(\eta))\exp\left(\frac{c}{2}\Lambda^*(\eta,\partial D)\right)\mu_\C(d\eta)=\int_{\{\eta\subset D'\}}F(\eta)\exp\left(\frac{c}{2}\Lambda^*(\eta,\partial D')\right)\mu_\C(d\eta)
	\end{equation*}
	for any positive measurable $F$ defined on loops contained in $D'$. Now we choose
	$$F(\eta)=\exp\left(\frac{c}{2}(\mathcal{W}(\phi^{-1}(\eta),\partial D)-\Lambda^*(\phi^{-1}(\eta),\partial D))\right)G(\eta).$$
	By \cite[Theorem 2.5]{carfagnini2024onsager} and the remark after it, we have
	$$F(\eta)=\exp\left(\frac{c}{2}(\mathcal{W}(\eta,\partial D')-\Lambda^*(\eta,\partial D'))\right)G(\eta).$$
	Hence we obtain
	\begin{equation*}
		\int_{\{\eta\subset D\}}G(\phi(\eta))\exp\left(\frac{c}{2}\mathcal{W}(\eta,\partial D)\right)\mu_\C(d\eta)=\int_{\{\eta\subset D'\}}F(\eta)\exp\left(\frac{c}{2}\mathcal{W}(\eta,\partial D')\right)\mu_\C(d\eta),
	\end{equation*}
	which is exactly that $\phi\circ\wh\mu_D=\wh\mu_{D'}$, as desired.
\end{proof}

\begin{remark}
$(\mu_D)$ and $(\wh \mu_D)$ are generally two different family of measures. Indeed, for any simply connected domain $D\subset\C$ and $c\in(0,1]$, $\mu_D$ (resp.\ $\wh\mu_D$) can be realized as the counting measure on outer boundaries of Brownian (resp.\ $\SLE_{8/3}$) loop soup clusters of intensity $\frac{c}{2}$ in $D$. This will be rigorously proved in a forthcoming work~\cite{clqsw}.
\end{remark}

In the following, we focus on the case $c\in(0,1]$.
We use a superscript 0 to denote the corresponding measure restricted to loops surrounding the origin.
As in Section~\ref{sec:general}, we will use the $\SLE_{8/3}$ loop soup observables to express quantities of $\mu_\C^0$. Let $\mathcal{L}^\cW_\C$ be the $\SLE_{8/3}$ loop soup on $\C$ of intensity $\frac{c}{2}$, and $\bP,\bE$ be the law and the expectation with respect to $\mathcal{L}^\cW_\C$. Recall the notation introduced above Theorem~\ref{thm:intersect}. In the following, for any function $f$ defined on some $\Omega\subset\C$, we let $f(A):=f(A\cap\Omega)$ for any $A\subset\C$.

\begin{proposition}\label{prop:priori}
	For any $0<r<1$, we have $\wh\mu_{\D}^0[\ell\not\subset r\D]<\infty$. Furthermore, for any simply connected $D\subset\D$ containing the origin, we have
	\begin{align}\label{eq:recursion-loop}
		\mu_\C^0[\ell\not\subset r\D,\ell\subset D]=\bE\otimes\wh\mu_\D^0[\ell\not\subset f_{D,\infty}( r\D)].
	\end{align}
\end{proposition}

To prove Proposition~\ref{prop:priori}, we use ideas similar to Section~\ref{sec:general}. First, we define an exploration process analogous to Definition~\ref{def:ep} as follows. Let $D_0:=D$, and $D_{n+1}\subset D_n$ be the connected component containing $0$ of the complement of the closure of the union of $\partial D_n$ and all loops in $\mathcal{L}^\cW_\C$ that intersect $\partial D_n$. 
Note that $D_\infty$ is an open set and has the origin as its interior point, as discussed above Theorem~\ref{thm:intersect}. Hence, each $D_n$ is well-defined and $D_\infty\subset D_n$.
We have the following result analogous to Lemma~\ref{lem:dH}. 

\begin{lemma}\label{lem:Dni}
$\bP$-a.s. $\lim_{n\to\infty}d_H(D_n, D_\infty)=0$, and $f_{D,n}^{-1}\to f_{D,\infty}^{-1}$ locally uniformly.
\end{lemma}

\begin{proof}
For the first claim, by the inversion invariance of $\mathcal{L}^\cW_\C$ (under $z\mapsto 1/z$) \cite{werner2008conformally}, it suffices to show that 
\begin{equation}\label{eq:Dni}
    \text{$\bP$-a.s.\ } \lim_{n\to\infty}d_H(F_n, F_\infty)=0,
\end{equation}
where $(F_n)_{n\ge0}$ is the exploration process such that $F_0:=D$, and for all $n\ge0$, $F_{n+1}$ is the filling of the union of $F_n$ and all loops in $\mathcal{L}^\cW_\C$ that intersect $F_n$, and $F_\infty$ is the filling of the union of $D$ and all clusters in $\mathcal{L}^\cW_\C$ that intersect $D$. Below, we will couple $\mathcal{L}^\cW_\C$ and the Brownian loop soup $\cL_{\C}$ in the plane such that $\mathcal{L}^\cW_\C$ is given by the collection of all outer boundaries of Brownian loops in $\cL_{\C}$. In fact, they are the same if we consider the fillings of loops in these two collections (this observation will be useful below). We recall the fact that for a deterministic domain $V$, the collection of loops in $\mathcal{L}^\cW_\C$ that are in $V$ is distributed as a $\SLE_{8/3}$ loop soup in $V$, which is a Poisson point process with intensity measure $\frac c2\cW\mid_V$.

We will reduce the analysis to some random bounded domain $U$ by using some basic facts about the $\SLE_{8/3}$ loop soup from \cite{kemppainen2016nested}.  
By inversion invariance and results in~\cite{sheffield2012conformal}, there a.s. exists a unique cluster $\Cout$ in $\mathcal{L}^\cW_\C$ that \emph{surrounds} $D$ (i.e.\ its filling contains $D$) and has the smallest and positive distance to $D$\footnote{Specifically, let $\hat D$ be the image of $\mathbb C\setminus \overline D$ under the map $z\mapsto 1/z$. 
By inversion, $\Cout$ will correspond to the outermost cluster surrounding $0$ in the $\SLE_{8/3}$ loop soup on $\hat D$ (once it exists and has has a positive distance to $\partial\hat D$). First, by~\cite[Section~2.1]{kemppainen2016nested}, such an outermost cluster in $\hat D$ exists, and its outer boundary has the same law as the outer boundary of the outermost cluster surrounding $0$ in the Brownian loop soup on $\hat D$. Furthermore, this outer boundary has a positive distance to $\partial\hat D$ almost surely by~\cite[Lemma~9.4]{sheffield2012conformal}. This gives the result.},
which is a.s. bounded. Let $U$ be the connected component containing $D$ of the complement of the closure of $\Cout$, which is a bounded domain since $\Cout$ is bounded. We denote by $\mathcal{L}^\cW_U$ the collection of loops in $\mathcal{L}^\cW_\C$ that are in $U$.
Then, 
as we will later show in Appendix~\ref{appd2}, we have
\begin{equation}\label{eq:loop_soup}
    \begin{array}{c}\text{$\mathcal{L}^\cW_{U}$ given $U$ is distributed as a $\SLE_{8/3}$ loop soup in $U$ conditioned }\\ 
    \text{ to have no cluster that surrounds $D$ without intersecting $D$.}
    \end{array}
\end{equation}

By definition, $(F_n)_{n\ge 0}$ and $F_\infty$ only depend on the collection of fillings of loops in $\mathcal{L}^\cW_{U}$, which is the same as the collection of fillings of Brownian loops in $\mathcal{L}_{U}$ ($\cL_{\C}$ restricted to loops in $U$). Hence, in order to get \eqref{eq:Dni}, we can first condition on $U$ and then consider the exploration process $(F_n)_{n\ge 0}$ as defined via the fillings of loops in the Brownian loop soup $\mathcal{L}_{U}$ instead.
Now, one can iterate the same proof as that of Lemma~\ref{lem:dH} by replacing $\D$ and $A$ with $U$ and $D$, respectively\footnote{More precisely, for any small $\eps>0$, let $(\cC_i)_{1\le i\le m}$ be the clusters in the Brownian loop soup $\mathcal{L}_{U}$ with diameter larger than $\eps$ that intersect $D$. By \cite[Theorem 4.1]{van2016random} again (which holds for any bounded simply connected domain), every cluster $\cC_i$, $1\le i\le m$, can be approximated by a sequence of finite subclusters $\cC_i^N$ of $\cC_i$ such that $\lim_{N\to\infty}d_H(\Fill(\cC_i^N),\Fill(\cC_i))=0$. Since each finite subcluster $\cC_i^N$ is contained in $F_n$ for sufficiently large $n$, we conclude $\lim_{n\to\infty}d_H(F_n, F_\infty)=0$.}, to conclude \eqref{eq:Dni}.
This finishes the first claim.

Once $D_n\to D_\infty$ in Hausdorff distance, we have $(D_n,0)\to(D,0)$ in the Carath\'eodory sense since $D_n$ is decreasing and $0\in D_\infty$. By Carath\'eodory kernel theorem (Theorem~\ref{thm:Caratheodory}), $f_{D,n}^{-1}\to f_{D,\infty}^{-1}$ locally uniformly.
\end{proof}

\begin{proof}[Proof of Proposition~\ref{prop:priori}]
Using \eqref{eq:hmu}, we obtain that
	\begin{align}\label{eq:rec-loop1}
		\mu_\C^0[\ell\not\subset r\D,\ell\subset D]=\wh\mu_D^0\Big[{\bf 1}_{\ell\not\subset  r\D} \exp\Big(-\frac{c}{2}\mathcal{W}(\ell,\partial D)\Big)\Big]
		=\bE\otimes\wh\mu_D^0[\ell\not\subset r\D,\ell\subset D_1].
	\end{align}  
    Similar to \eqref{eq:KA1}, in order to get the last equality, we also need to use the local finiteness property of $\mathcal{L}^\cW_\C$, that is, the number of loops in $\mathcal{L}^\cW_\C$ that intersect $D$ and have diameter at least $\eps$ is a.s. finite. To see this, it suffices to note that the $\cW$-mass of the loops intersecting $D$ and of diameter at least $\eps$ is finite.\footnote{Indeed, for some large fixed $R>0$ such that $R\D\supset D$, note that both $\cW(R\S^1,\partial D)$ and the $\cW$-mass of the loops contained in $R\D$ with diameter at least $\eps$ are finite (the former is due to e.g.~\cite[Page 841, the third paragraph]{kemppainen2016nested}, and the latter is by the local finiteness of the Brownian loop measure shown before).}
    
Let $f_{D,n}$ be the conformal map from $D_n$ to $\D$ such that $f_{D,n}(0)=0$ and $f'_{D,n}(0)>0$. 
	Iterating \eqref{eq:rec-loop1}, we conclude that for all $n$, 
    \begin{equation}\label{eq:rec-loop2}
    \begin{aligned}
		\mu_\C^0[\ell\not\subset r\D,\ell\subset D]&=\bE\otimes\wh\mu_{D_n}^0[\ell\not\subset r\D,\ell\subset D_{n+1}]\\&=\bE\otimes\wh\mu_{\D}^0[\ell\not\subset f_{D,n}( r\D),\ell\subset f_{D,n}(D_{n+1})],
	\end{aligned}
    \end{equation}
	where we used the conformal invariance of $\wh\mu_{\D}$ from Lemma~\ref{lem:ci-hmu}.

By Lemma~\ref{lem:Dni}, $f_{D,n}^{-1}\to f_{D,\infty}^{-1}$ in the local uniform topology. In particular, we have
	\begin{align}\label{eq:rec-loop5}
		{\bf 1}_{\ell\not\subset f_{D,n}( r\D), \ell\subset f_{D,n}(D_{n+1})}\to{\bf 1}_{\ell\not\subset f_{D,\infty}( r\D)},\quad\bE\otimes\wh\mu_{\D}^0-{\rm a.e.} 
	\end{align}
	Then, by Fatou's Lemma, taking $n\to\infty$ in \eqref{eq:rec-loop2}, we obtain that for any $0<r<1$,
	\begin{align}\label{eq:rec-loop6}
		\infty>\mu_\C^0[\ell\not\subset r\D,\ell\subset D]\ge\bE\otimes\wh\mu_\D^0[\ell\not\subset f_{D,\infty}( r\D)].
	\end{align}
	
Next, we will show that \eqref{eq:rec-loop6} implies $\wh\mu_{\D}^0[\ell\not\subset r\D]<\infty$. First note that $\bP$-a.s. $\D_\infty$ contains $0$ as an interior point. This implies that there exists an $\varepsilon_0\in (0,1)$ such that $\D_\infty\supset\varepsilon_0\D$ with a positive probability $p_0$ under $\bP$. Then, we show that  $\D_\infty\supset\varepsilon_0\D$ implies $f_{\D,\infty}( \eps_0r\D)\subset r\D$. 
Noting that $z\mapsto f_{\D,\infty}(\varepsilon_0z)$ defines a holomorphic function on $\D$ taking values in $\D$, we get $|f_{\D,\infty}(\varepsilon_0 z)|\le |z|$ from Schwarz lemma. In particular, for any $z\in \eps_0r\D\subset \D_\infty$, it holds that $|f_{\D,\infty}(z)|\le |\varepsilon_0^{-1}z|\le r$, and $f_{\D,\infty}( \eps_0 r\D)\subset  r\D$. Therefore,
\begin{align*}
p_0\,\wh\mu_\D^0[\ell\not\subset r\D] \le \bE\otimes\wh\mu_\D^0[\ell\not\subset f_{\D,\infty}( \eps_0r\D)].
\end{align*}
Since~\eqref{eq:rec-loop6} holds for any $r\in(0,1)$, it also holds with $r$ replaced by $\varepsilon_0 r$, which implies the RHS above is finite, and we obtain $\wh\mu_{\D}^0[\ell\not\subset r\D]<\infty$ as desired. It also follows that $\bE\otimes\wh\mu_{\D}^0[\ \cdot\ ;\ \ell\not\subset r\D]$ is a finite measure for any fixed $r$, since its total mass is just equal to $\wh\mu_{\D}^0[\ell\not\subset r\D]<\infty$.
	
	It remains to show~\eqref{eq:recursion-loop}. We go back to~\eqref{eq:rec-loop2}. By Schwarz lemma again, $D_n\subset\D$ gives $ r\D\subset f_{D,n}( r\D)$. Hence, we can rewrite~\eqref{eq:rec-loop2} as
	\begin{align}\label{eq:rec-loop3}
		\mu_\C^0[\ell\not\subset r\D,\ell\subset D]=\bE\otimes\wh\mu_{\D}^0[\ell\not\subset f_{D,n}( r\D),\ell\subset f_{D,n}(D_{n+1});\ \ell\not\subset r\D].
	\end{align}
	Then, taking $n\to\infty$ in \eqref{eq:rec-loop3}, by~\eqref{eq:rec-loop5} and the dominated convergence theorem (using that $\bE\otimes\wh\mu_{\D}^0[\ \cdot\ ;\ \ell\not\subset r\D]$ is a finite measure), we finally obtain~\eqref{eq:recursion-loop}.
\end{proof}

In the following, we first provide several basic estimates on $\mu_\C^0$ and $\wh\mu_\D^0$ in Section~\ref{sec:apriori}, and then finish the proof of Theorems~\ref{thm:loop-unique} and~\ref{thm:intersect} in Section~\ref{subsec:hitting}.

\subsection{A priori estimates on $\mu_\C^0$ and $\wh\mu_\D^0$}\label{sec:apriori}

Let $\|\ell\|_\infty:=\max_{z\in\ell}|z|$ for a simple loop $\ell\subset\C$.

\begin{lemma}\label{lem:c-mass}
For any $\mu_\C$ with parameter $c\in[0,1]$, there is a $\lambda\in(0,\infty)$ such that for any $a>b>0$, $\mu_\C^0(\|\ell\|_\infty\in[a,b))=\lambda\log\frac{b}{a}$.
\end{lemma}
\begin{proof}
Since $\mu_\C$ is \emph{non-trivial} (see the paragraph above Definition~\ref{def:MKS}), we have $U(a,b):=\mu_\C^0(\|\ell\|_\infty\in[a,b))\in (0,\infty)$. Then $U(a,b)+U(b,c)=U(a,c)$ for any $0<a<b<c$. Furthermore, the scaling invariance of $\mu_\C^0$ gives $U(a,b)=U(ka,kb)$ for all $k>0$. Then $x\mapsto U(1,x)$ is increasing with $U(1,xy)=U(1,x)+U(1,y)$; thus $U(1,x)=\lambda\log x$ for some $\lambda\in (0,\infty)$.
\end{proof}

Note that $\lambda$ depends on the specific choice of $\mu_\C$ as mentioned. However, we can determine the following exact value of $\lambda$ for $c\in(0,1]$ when $\mu_\C$ is equal to the counting measure on full-plane $\CLE_\kappa$ loops.

\begin{corollary}\label{cor:lambda0}
	For $c\in (0,1]$, let $\mu_\C=\SLE_\kappa^\lp$, i.e. the counting measure on full-plane $\CLE_\kappa$ loops, then the $\lambda$ in Lemma~\ref{lem:c-mass} is given by $\lambda=\frac{1}{\pi}(\frac{\kappa}{4}-1)\cot(\pi(1-\frac{4}{\kappa}))$.
\end{corollary}
\begin{proof}
	Let $\Gamma$ be sampled from the full-plane $\CLE_\kappa$, and $N_C$ be the number of loop $\ell$'s in $\Gamma$ surrounding the origin with $\|\ell\|_\infty\in[1,e^C)$. By~\cite[Lemma 9.2]{ang2024sle} (which is based on~\cite{schramm2009conformal}), we have $\frac{\E[N_C]}{C}\to\frac{1}{\pi}(\frac{\kappa}{4}-1)\cot(\pi(1-\frac{4}{\kappa}))$ as $C\to\infty$. Comparing with Lemma~\ref{lem:c-mass}, the result then follows.
\end{proof}

We also mention that for $c=0$ and $\mu_{\C,c=0}$ replaced by Werner's measure $\mathcal{W}$, the corresponding $\lambda$ in Lemma~\ref{lem:c-mass} equals $\frac{\pi}{5}$~\cite[Page 151]{werner2008conformally}.

\begin{lemma}\label{lem:83mass}
	There is an $\alpha>0$ such that for any $0<\eps<\frac{1}{2}$, $\mathcal{W}(\eps\D,\S^1)=O(\eps^\alpha)$.
\end{lemma}
\begin{proof}
	This basically follows from \cite{NW11}. We add some details here for completeness. First, by \cite[Lemma 2]{NW11} and scaling invariance of $\cW$, we have 
	\begin{equation}\label{eq:area}
		\cW(\eps\D,\S^1)\le\int_0^{2\eps} \frac{1}{r}\, \E(\mathrm{Area}(\{z: \dist(z,\gamma)\le r\}))\, dr,
	\end{equation}
	where $\mathrm{Area}$ denotes the Lebesgue measure on $\R^2$, and $\gamma$ under $\E$ has the law of the outer boundary of the Brownian loop of time-length $1$. By the proof of \cite[Lemma 4]{NW11}, the expected area of $r$-neighborhood of $\gamma$ is $O(r^\alpha)$ for some constant $\alpha>0$. This concludes the proof from \eqref{eq:area}.
\end{proof}

In fact, the optimal $\alpha$ in Lemma~\ref{lem:83mass} is $\frac23$, which can be deduced from the Brownian disconnection exponents. However, we will not need such an explicit bound.

\begin{proposition}\label{prop:dc}
For any $\mu_\C$ with parameter $c\in (0,1]$, there is a constant $\zeta_1\in (0,\infty)$ such that $\wh\mu_\D^0(\|\ell\|_\infty\in[\varepsilon,1))=\lambda|\log\varepsilon|+\zeta_1+o(1)$ as $\varepsilon\to0$.
\end{proposition}
\begin{proof}
	By Proposition~\ref{prop:priori}, $\wh\mu_\D^0(\|\ell\|_\infty\in[\varepsilon,1))<\infty$.
	By \eqref{eq:hmu}, we have 
	\begin{align*}
		0\le\wh\mu_\D^0(\|\ell\|_\infty\in[\varepsilon,2\eps))\!-\!\mu_\C^0(\|\ell\|_\infty\in[\varepsilon,2\eps)) \le \mu_\C^0(\|\ell\|_\infty\in[\varepsilon,2\eps))[\exp(\frac{c}{2}\cW(2\eps\D,\S^1))\!-\!1]
	\end{align*}
    which is $O(\eps^\alpha)$ according to Lemmas~\ref{lem:c-mass} and~\ref{lem:83mass}.
	Hence, $\wh\mu_\D^0(\|\ell\|_\infty\in[\varepsilon,1))-\mu_\C^0(\|\ell\|_\infty\in[\varepsilon,1))$ converges to some $\zeta_1\in(0,\infty)$ as $\varepsilon\to0$.
\end{proof}

\begin{corollary}\label{cor:83mass-scale}
	Let $\beta\in[0,\frac{1}{2}]$. We have
	\begin{equation*}
		\wh\mu_\D^0(\|\ell\|_\infty\in[\varepsilon,a\varepsilon))=\lambda\log a+o(1) \quad \text{as} \quad \eps\to 0,
	\end{equation*}
	uniformly for all $1\le a\le \varepsilon^{-\beta}$, where $o(1)$ is independent of $\beta$.
\end{corollary}
\begin{proof}
	Write $\wh\mu_\D^0(\|\ell\|_\infty\in[\varepsilon,a\varepsilon))=\wh\mu_\D^0(\|\ell\|_\infty\in[\varepsilon,1))-\wh\mu_\D^0(\|\ell\|_\infty\in[a\varepsilon,1))$. The result then follows from Proposition~\ref{prop:dc}.
\end{proof}

\subsection{The proof of Theorems~\ref{thm:loop-unique} and~\ref{thm:intersect}}\label{subsec:hitting}

This section is dedicated to the proof of Theorem~\ref{thm:intersect}, which implies the uniqueness of the MKS measure (Theorem~\ref{thm:loop-unique}) immediately.
By scaling invariance of MKS measures, it suffices to prove Theorem~\ref{thm:intersect} for simply connected domains $U,D$ such that $\{0\}\subset U\subset D\subset\D$. Therefore, we will assume $D\subset\D$ throughout this section.

We first define the following important quantity $I(D)$ whenever the limit exists (which is the case as we will see immediately)
\begin{equation}\label{eq:ID}
	I(D):=\lim_{\varepsilon\to0}\big(\wh\mu_\D^0[\ell\not\subset \varepsilon\D]-\mu_\C^0[\ell\not\subset\varepsilon\D,\ell\subset D]\big).
\end{equation}
Note in particular that $I(\D)=\zeta_1$ according to Proposition~\ref{prop:dc}.
We will show that $I(D)$ has a nice loop-soup expression (see Proposition~\ref{prop:elog0}), which is the key to Theorem~\ref{thm:intersect}. 
Before going any further, we first show that the limit in \eqref{eq:ID} exists by
establishing an equality between $I(D)$ and the $\mu_\C^0$-mass of loops in $\D$ that are not fully in $D$.

\begin{lemma}\label{lem:ID1}
	The limit in~\eqref{eq:ID} exists and $I(D)=\mu_\C^0[\ell\not\subset D,\ell\subset \D]+\zeta_1\in (0,\infty)$.
\end{lemma}
\begin{proof}
	Since $D\subset\D$, we have
	\begin{align*}
		\mu_\C^0[\ell\not\subset D,\ell\subset \D]&=\lim_{\varepsilon\to0}\big(\mu_\C^0[\ell\not\subset \varepsilon\D,\ell\subset \D]-\mu_\C^0[\ell\not\subset\varepsilon\D,\ell\subset D]\big)\\
		&=\lim_{\varepsilon\to0}\big(\wh\mu_\D^0[\ell\not\subset \varepsilon\D]-\zeta_1-\mu_\C^0[\ell\not\subset\varepsilon\D,\ell\subset D]\big),
	\end{align*}
	where the second line is from Proposition~\ref{prop:dc}.
	Hence $I(D)=\mu_\C^0[\ell\not\subset D,\ell\subset \D]+\zeta_1$. Note that $\mu_\C^0[\ell\not\subset D,\ell\subset \D]\in (0,\infty)$ thanks to the non-triviality of $\mu_\C$.
\end{proof}

The following proposition provides a clean form of $I(D)$ that only involves randomness from the loop soup.

\begin{proposition}\label{prop:elog0}
	We have $I(D)=\lambda\bE\log f_{D,\infty}'(0)$. In particular, $\zeta_1=I(\D)=\lambda\bE\log f_{\D,\infty}'(0)$.
\end{proposition} 

We first show how to derive Theorem~\ref{thm:intersect} (and hence Theorem~\ref{thm:loop-unique}) from Proposition~\ref{prop:elog0}.

\begin{proof}[Proof of Theorem~\ref{thm:intersect}, assuming Proposition~\ref{prop:elog0}]
	Write
	$\mu_\C^0[\ell\not\subset U,\ell\subset D]= \mu_\C^0[\ell\not\subset U,\ell\subset \D] - \mu_\C^0[\ell\not\subset D,\ell\subset \D]$. We finish the proof by applying  Lemma~\ref{lem:ID1} and Proposition~\ref{prop:elog0}. 
\end{proof}
\begin{proof}[Proof of Theorem~\ref{thm:loop-unique}]
	From Theorem~\ref{thm:intersect}, we have determined $\mu_\C^0[\ell \not\subset U,\ell \subset D]$ for all simply connected domains $U$ and $D$ satisfying $U \subset D$. This indeed characterizes $\mu_\C$ by a rather straightforward measure-theoretical argument, as presented in~\cite[Page 146]{werner2008conformally}.
\end{proof}

In the remaining of this section, we aim to show Proposition~\ref{prop:elog0}. Our proof relies crucially on the following loop-soup interpolation of $I(D)$.

\begin{lemma}\label{lem:Ils}
	We have
	\begin{equation}\label{eq:ID-rec}
		I(D)=\lim_{\varepsilon\to0}\bE\otimes\wh\mu_\D^0[\ell\subset f_{D,\infty}(\varepsilon\D)\setminus(\varepsilon\D)].
	\end{equation}
\end{lemma}
\begin{proof}
	Since $D_\infty\subset\D$, by Schwarz lemma, we have $f_{D,\infty}(\varepsilon\D)\supset\varepsilon\D$. Then by Proposition~\ref{prop:priori},
	\begin{align*}
		\wh\mu_\D^0[\ell\not\subset \varepsilon\D]-\mu_\C^0[\ell\not\subset\varepsilon\D,\ell\subset D]&=\bE\otimes\wh\mu_\D^0[\ell\not\subset \varepsilon\D]-\bE\otimes\wh\mu_\D^0[\ell\not\subset f_{D,\infty}(\varepsilon\D)]\\
		&=\bE\otimes\wh\mu_\D^0[\ell\subset f_{D,\infty}(\varepsilon\D)\setminus(\varepsilon\D)].
	\end{align*}
	Taking $\varepsilon\to0$, we get~\eqref{eq:ID-rec}. 
\end{proof}

We will use the following technical lemma to control the distortion of conformal maps. 

\begin{lemma}\label{lem:distortion}
	Let $D$ be a simply connected domain in $\D$ that contains $0$. Let $f$ be the conformal map from $D\to\D$ with $f(0)=0$ and $f'(0)>0$. Suppose $\eps,\beta \in (0,1)$, $\eps^{1-\beta}\le \frac{1}{16}$ and $f'(0)<\eps^{-\beta}$. There is a universal constant $C$ such that for any $z$ with $|z|=\eps$,
	\[
	f'(0)\varepsilon(1-C\varepsilon^{1-\beta})\le|f(z)|\le f'(0)\varepsilon(1+C\varepsilon^{1-\beta}).
	\]
\end{lemma}
\begin{proof}
	By Koebe's 1/4 theorem, we have $r:=\dist(0,\partial D)\ge \frac{1}{4}f'(0)^{-1}>\frac{1}{4}\eps^\beta$. Consider the conformal map $g(w):=f(rw)/(rf'(0))$ on $\D$, which satisfies $g'(0)=1$. By~\cite[Proposition 3.26]{lawler2008conformally}, for any $|w|=\frac{\eps}{r}\le 4\eps^{1-\beta}\le \frac{1}{4}$, we have 
	$|g(w)-w|\le C \,|w|^2$ for some universal constant $C$.
	Setting $z=rw$, we have $|f(z)-f'(0)z|\le Cf'(0)\eps^{2-\beta}$, concluding the proof.
\end{proof}

\begin{lemma}\label{lem:IDs}
	For any $\alpha_0\in(0,\frac12]$, we have 
	\begin{equation}\label{eq:ids}
		\lim_{\varepsilon\to0}\bE\otimes\wh\mu_\D^0[\ell\subset f_{D,\infty}(\varepsilon\D)\setminus(\varepsilon\D);\ f_{D,\infty}'(0)<\varepsilon^{-\alpha_0}]=\lambda\,\bE\log f_{D,\infty}'(0).
	\end{equation}
	Therefore,
	\begin{equation}\label{eq:finit}
		\lambda\,\bE\log f_{D,\infty}'(0)\le I(D)<\infty.
	\end{equation}
\end{lemma}
\begin{proof}
	By Lemma~\ref{lem:distortion}, we have
	\begin{align*}
		&\Big|\bE\otimes\wh\mu_\D^0[\ell\subset f_{D,\infty}(\varepsilon\D)\setminus(\varepsilon\D);\ f_{D,\infty}'(0)<\varepsilon^{-\alpha_0}]\\&-\bE\otimes\wh\mu_\D^0[\|\ell\|_\infty\in[\varepsilon, f_{D,\infty}'(0)\varepsilon);\ f_{D,\infty}'(0)<\varepsilon^{-\alpha_0}]\Big|\\
		\le\ & \bE\otimes\wh\mu_\D^0[\|\ell\|_\infty\in[f_{D,\infty}'(0)\varepsilon(1-C\varepsilon^{1-\alpha_0}), f_{D,\infty}'(0)\varepsilon(1+C\varepsilon^{1-\alpha_0});\ f_{D,\infty}'(0)<\varepsilon^{-\alpha_0}]\\
		\le\ & O(\eps^{1-\alpha_0})\,\bP(f_{D,\infty}'(0)<\varepsilon^{-\alpha_0}) \to 0 \quad \text{as} \quad \eps\to 0,
	\end{align*}
	where we used Corollary~\ref{cor:83mass-scale} to estimate the $\wh\mu_\D^0$-mass of $\ell$ in the second last line. 
	It follows that
	\begin{align*}
		& \lim_{\eps\to0}\bE\otimes\wh\mu_\D^0[\|\ell\|_\infty\in[\varepsilon, f_{D,\infty}'(0)\varepsilon);\ f_{D,\infty}'(0)<\varepsilon^{-\alpha_0}]\\
		=\ &\lim_{\eps\to0} \bE[\lambda\,\log f_{D,\infty}'(0)+o(1);\ f_{D,\infty}'(0)<\varepsilon^{-\alpha_0}]
		&\quad\text{(by Corollary~\ref{cor:83mass-scale})}\\
		=\ &\lambda\,\bE\log f_{D,\infty}'(0) &\quad\text{(the $o(1)$ is independent of $\bE$)}.
	\end{align*}
	Combined, we complete the proof of \eqref{eq:ids}.	
Now combining~\eqref{eq:ids} and Lemma~\ref{lem:Ils}, we obtain that
    \[\lambda\,\bE\log f_{D,\infty}'(0)\le\lim_{\varepsilon\to0}\bE\otimes\wh\mu_\D^0(\ell\subset f_{D,\infty}(\varepsilon\D)\setminus(\varepsilon\D))=I(D).\] 
    By Lemma~\ref{lem:ID1}, $I(D)<\infty$, and~\eqref{eq:finit} then follows.
\end{proof}

Finally, we turn to 

\begin{proof}[Proof of Proposition~\ref{prop:elog0}]
	By \eqref{eq:ids} and Lemma~\ref{lem:Ils},
	it remains to show that
	\begin{equation*}
		\lim_{\eps\to0}\bE\otimes\wh\mu_\D^0[\ell\subset f_{D,\infty}(\varepsilon\D)\setminus(\varepsilon\D);\ f_{D,\infty}'(0)\ge\varepsilon^{-\alpha_0}]=0.
	\end{equation*}
	First, by \eqref{eq:finit}, as $\varepsilon\to0$,
	\[
	|\log\varepsilon|\bP[f_{D,\infty}'(0)\ge\varepsilon^{-\alpha_0}]\le\frac{1}{\alpha_0}\bE[\log f_{D,\infty}'(0),\ f_{D,\infty}'(0)>\varepsilon^{-\alpha_0}]\to0.
	\]
	Therefore, by Proposition~\ref{prop:dc}, we have
	\begin{align*}
		\bE\otimes\wh\mu_\D^0[\ell\subset f_{D,\infty}(\varepsilon\D)\!\setminus\!(\varepsilon\D);\ f_{D,\infty}'(0)\ge\varepsilon^{-\alpha_0}]&\le \bE\otimes\wh\mu_\D^0[\|\ell\|_\infty\in[\varepsilon,1);\ f_{D,\infty}'(0)\ge\varepsilon^{-\alpha_0}]\\
		&=(\lambda|\log\varepsilon|\!+\!\zeta_1\!+\!o(1))\bP[f_{D,\infty}'(0)\ge\varepsilon^{-\alpha_0}],
	\end{align*}
	which tends to $0$ as $\varepsilon\to0$. The proof is now complete.    
\end{proof}

\section{Electrical thickness}\label{sec:ele-thik}

This section is to establish Theorem~\ref{thm:electrical-thickness} from Theorem~\ref{thm:intersect}. We inherit the notation from Section~\ref{sec:loop}.
We first show that the $\mu_\C^0$-mass of loops that intersect the unit circle $\S^1$ is finite.
\begin{lemma}\label{lem:finiteness}
	For any $c\in (0,1]$, we have $\zeta_2:=\mu_\C^0(\ell\cap\S^1\neq\emptyset)<\infty$.
\end{lemma}
\begin{proof}
	According to the uniqueness of the MKS measure (Theorem~\ref{thm:loop-unique}), it suffices to prove the statement for $\mu_\C=\SLE_\kappa^\lp$, the counting measure on the full-plane $\CLE_\kappa$ configuration $\Gamma$.
	First, there is a universal $u>0$ such that for any simply connected domain $D$ containing $0$ with $\dist(0,\partial D)<1$, with probability at least $u$, the outermost $\CLE_\kappa$ loop in $D$ surrounding $0$ is contained in $\frac{1}{2}\D$  (this is directly from the loop-soup construction of $\CLE_\kappa$~\cite{sheffield2012conformal}; see e.g. \cite[Section 3.1, Fact 4]{kemppainen2016nested}). Then by the Markov property of $\CLE_\kappa$, the number of loops in $\Gamma$ intersecting $\S^1$ is stochastically dominated by a geometric distribution with parameter $u$, i.e., $\mu_\C^0(\ell\cap\S^1\neq\emptyset)\le \sum_{n\ge1}nu(1-u)^{n-1}<\infty$.
\end{proof}

Denote $\A_{\varepsilon,R}:=\{z:\varepsilon<|z|< R\}$. The following gives the $\mu_\C^0$-mass of loops that intersect $\A_{\varepsilon,R}$. 
\begin{lemma}\label{lem:c0}
	For any $c\in (0,1]$, let
	\begin{equation}\label{eq:zeta}
		\zeta:=\zeta_2-2\zeta_1=\mu_\C^0(\ell\cap\S^1\neq\emptyset)-2\lambda\bE\log f_{\D,\infty}'(0) \in(-\infty,\infty),
	\end{equation}
	where $\zeta_1$ and $\zeta_2$ are defined in Proposition~\ref{prop:dc} and Lemma~\ref{lem:finiteness} respectively. Then, we have
	\begin{equation}\label{eq:zeta2}
		\lim_{\varepsilon\to0, R\to\infty}\big(\mu_\C^0(\ell\cap \A_{\varepsilon,R}\neq\emptyset)-\wh\mu_\D^0(\|\ell\|_\infty\in(\varepsilon,1))-\wh\mu_\D^0(\|\ell\|_\infty\in(R^{-1},1))\big)=\zeta.
	\end{equation}
\end{lemma}
\begin{proof}
	The last equation of \eqref{eq:zeta} follows from Proposition~\ref{prop:elog0}.
	To show \eqref{eq:zeta2}, write $\|\ell\|_{\min}:=\min_{z\in\ell}|z|$, and we decompose the mass of loops as 
	\begin{equation}\label{eq:decom}
		\mu_\C^0(\ell\cap \A_{\varepsilon,R}\neq\emptyset)= \mu_\C^0(\|\ell\|_\infty\in(\varepsilon,1))+\mu_\C^0(\|\ell\|_{\min}\in(1,R))+\mu_\C^0(\ell\cap\S^1\neq\emptyset).
	\end{equation}
	Note that $\mu_\C^0(\|\ell\|_{\min}\in(1,R))=\mu_\C^0(\|\ell\|_\infty\in(R^{-1},1))$ by inversion invariance. Using Proposition~\ref{prop:dc} to estimate the first two terms on the right-hand side of \eqref{eq:decom}, we finish the proof.
\end{proof}

Next, we consider general simply connected domain $D\subset\C$ containing $0$ (without the assumption that $D\subset\D$). The definition of $I(D)$ given in \eqref{eq:ID} can be extended to this general case straightforwardly. 
We can generalize the result in Proposition~\ref{prop:elog0} to $D\subset\C$ by scaling.

\begin{proposition}\label{prop:elog}
	For any simply connected domain $D\subset\C$ containing $0$, we have
	\begin{align}\label{eq:elog}
		I(D)=\lambda\,\bE\log f_{D,\infty}'(0).
	\end{align}
\end{proposition}
\begin{proof}
	Let $a:=\sup_{z\in D}|z|$ and $\wt D:=a^{-1}D$. By~\eqref{eq:ID} and scaling invariance of $\mu_\C^0$, we have $I(D)=I(\wt D)-\lambda\log a$. Furthermore, since $\wt D\subset\D$, Proposition~\ref{prop:elog0} gives $I(\wt D)=\lambda\bE\log f_{\wt D,\infty}'(0)$. Moreover, by scaling invariance of the whole-plane $\SLE_{8/3}$ loop-soup, $f_{D,\infty}'(0)$ has the same law as $a^{-1}f_{\wt D,\infty}'(0)$. Combined, we obtain $I(D)=\lambda\bE\log f_{\wt D,\infty}'(0)-\lambda\log a=\lambda\bE\log f_{D,\infty}'(0)$, as desired.
\end{proof}

Now we complete the proof of Theorem~\ref{thm:electrical-thickness}. Let $V$ be any connected compact set separating $0$ and $\infty$, and recall the notation introduced above Theorem~\ref{thm:electrical-thickness}. For clarity, we fix the conformal maps $f_{V_\infty}$ and $h_{V_\infty}$ such that $f_{V_\infty}'(0),h_{V_\infty}'(\infty)>0$. Let $\Omega_V$ and $\Omega_V^*$ be the connected component of $\C\setminus V$ containing $0$ and $\infty$, respectively. 
\begin{proof}[Proof of Theorem~\ref{thm:electrical-thickness}]
	By Proposition~\ref{prop:elog} and inversion invariance, we have
    \begin{align}
    I(\Omega_{V})=\lim_{\varepsilon\to0}\left(\wh\mu_\D^0(\ell\not\subset\varepsilon\D)-\mu_\C^0(\ell\not\subset\varepsilon\D,\ell\subset\Omega_{V})\right)&=\lambda\bE\log f_{V_\infty}'(0),\label{eq:I1}\\
		\lim_{R\to\infty}\left(\wh\mu_\D^0(\ell\not\subset R^{-1}\D)-\mu_\C^0(\ell\not\subset R\D^*,\ell\subset\Omega_{V}^*)\right)&=-\lambda\bE\log h_{V_\infty}'(\infty).\label{eq:I2}
    \end{align}
Moreover, for any $0<\varepsilon<R<\infty$ such that $V\subset\A_{\varepsilon, R}$, we have
\begin{equation}\label{eq:break}
	\begin{aligned}
		\mu_\C^0(\ell\cap V\neq\emptyset)
		&=\mu_\C^0(\ell\cap \A_{\varepsilon, R}\neq\emptyset)-\mu_\C^0(\ell\not\subset\varepsilon\D,\ell\subset\Omega_{V})-\mu_\C^0(\ell\not\subset R\D^*,\ell\subset\Omega_{V}^*)\\
		&=\left(\wh\mu_\D^0(\ell\not\subset\varepsilon\D)-\mu_\C^0(\ell\not\subset\varepsilon\D,\ell\subset\Omega_{V})\right)\\&\quad\quad+\left(\wh\mu_\D^0(\ell\not\subset R^{-1}\D)-\mu_\C^0(\ell\not\subset R\D^*,\ell\subset\Omega_{V}^*)\right)\\
		&\quad\quad+\left(\mu_\C^0(\ell\cap \A_{\varepsilon, R}\neq\emptyset)-\wh\mu_\D^0(\ell\not\subset\varepsilon\D)-\wh\mu_\D^0(\ell\not\subset R^{-1}\D)\right).
	\end{aligned}
\end{equation}
	Taking $\varepsilon\to0$ and $R\to\infty$, using \eqref{eq:I1}, \eqref{eq:I2} and Lemma~\ref{lem:c0} to deal with the three terms on the right-hand side above respectively, we conclude that
	\begin{equation*}
		\mu_\C^0(\ell\cap V\neq\emptyset)=\lambda\bE\log f_{V_\infty}'(0)-\lambda\bE\log h_{V_\infty}'(\infty)+\zeta=\lambda\bE\left[\theta(V_\infty)\right]+\zeta,
	\end{equation*}
	which finishes the proof.
\end{proof}

\begin{remark}\label{rmk:c0}
We can get a counterpart of Theorem~\ref{thm:electrical-thickness} for $c=0$ similarly. Namely,
\begin{equation}\label{eq:8/3-et}
\mu_{\C,c=0}^0(\ell\cap V\neq\emptyset)=\lambda\theta(V)+ \mu_{\C,c=0}^0(\ell\cap \S^1\neq\emptyset),
\end{equation}
where $\lambda$ is provided by Lemma~\ref{lem:c-mass} 
for $c=0$, $\theta(V)$ is the electrical thickness of $V$, and $\mu_{\C,c=0}^0(\ell\cap \S^1\neq\emptyset)$ is finite due to~\cite[Lemmas~2 and~4]{NW11}.
To get~\eqref{eq:8/3-et}, it suffices to show for $V\subset\D$ by scaling. According to~\cite[Lemma 4]{werner2008conformally},
\[
\mu_{\C,c=0}^0(\ell\not\subset D, \ell\subset\D)=\lambda\log f_{D}'(0)\quad \text{for any simply connected } D\subset\D \text{ containing } 0.
\]
Then, \eqref{eq:8/3-et} follows by repeating~\eqref{eq:break} and taking the limit $\varepsilon\to0, R\to\infty$ (note that for $c=0$, $\wh\mu_\D^0(\cdot)=\mu_\C^0(\cdot{\bf 1}_{\cdot\subset\D})$).
    We remark that~\eqref{eq:8/3-et} looks similar to~\cite[Corollary 4.8]{wang2024brownian}, which computes the total mass of Brownian loops with winding $m\in\mathbb{Z}_+$ around $0$ and hitting $V$.
	
	As mentioned before, if $\mu_{\C,c=0}$ is Werner's measure $\mathcal{W}$,
	then $\lambda=\frac{\pi}{5}$~\cite{werner2008conformally}. One can also determine that $\mathcal{W}^0(\ell\cap \S^1\neq\emptyset)=\mathcal{W}(\ell\ {\rm surrounds}\ 0,\ \ell\cap \S^1\neq\emptyset)=\frac{\sqrt{6}}{15}$.\footnote{Let $Z_{8/3}(\tau)$ be the total mass of $\mathcal{W}$ on non-contractible loops in $\{z:e^{-2\pi\tau}<|z|< 1\}$. Then $\mathcal{W}(\ell\ {\rm surrounds}\ 0,\ \ell\cap \S^1\neq\emptyset)=\lim_{\tau\to\infty}Z_{8/3}(2\tau)-2Z_{8/3}(\tau)$. The explicit formula of $Z_{8/3}(\tau)$, up to a multiplicative constant, is given by Cardy $Z_{\rm Cardy}(\tau)$~\cite[(5)]{cardy2006n} and proved by~\cite{ang2022moduli}. By a forthcoming work~\cite{clqsw}, this multiplicative constant can be further specified, i.e., $Z_{8/3}(\tau)=\frac{3\sqrt{2}}{5}Z_{\rm Cardy}(\tau)$.}
\end{remark}

\section{Discussions and Possible Extensions}\label{sec:further}

At the end of this paper we briefly list several remarks and open questions.

\medskip
\noindent\textbf{General Riemann Surfaces.} Both the Brownian loop measure and the SLE loop measure can be defined on a general Riemann surface $\Sigma$, see e.g.~\cite{werner2008conformally, lawler2011defining, zhan2021sle}. It is interesting to see how our approach in Section~\ref{sec:loop} can be extended to the SLE loop measure $\mu_\Sigma$ on $\Sigma$. For instance, one can consider the $\SLE_{8/3}$ loop soup on $\Sigma$, and relate it to the quantity $\mu_\Sigma[\ell\subset A \text{ and non-contractible in } A]$ for some annular region $A\subset\Sigma$.

Note that recently Wang and Xue~\cite{wang2024brownian} demonstrated that there is a deep connection between the Brownian loop measure and the length spectrum of geodesics on $\Sigma$. Hence, it is natural to ask that on $\Sigma$ how the SLE loop measure relates to such geometric quantities.

\medskip
\noindent\textbf{The regime $c<0$ or $\kappa\in(0,\frac{8}{3})$.} Our approach relies on the interpretation of the loop-mass term in the generalized conformal restriction through the Brownian loop soup, which indeed utilizes the coupling from $\SLE_{8/3}$ to $\SLE_\kappa$ for $\kappa > \frac{8}{3}$. It would be interesting to see whether a similar approach can be adapted to the regime $c < 0$. 
This is connected to an alternative characterization of $\SLE_\kappa$ for $\kappa<\frac{8}{3}$ as the unique simple curve with the property that, when augmented with a specific density of Brownian loops, the resulting hull coincides with the union of certain Brownian motions~\cite{werner2003sles}.

\medskip
\noindent\textbf{Integrability of the SLE loop.} Our results, Theorems~\ref{thm:intersect} and~\ref{thm:electrical-thickness}, establish relationships between natural quantities of the $\SLE_\kappa$ loop measure and the $\SLE_{8/3}$ loop soup. A natural follow-up question is whether these quantities can be expressed as \emph{exact functions}, and whether further integrability for the $\SLE_\kappa$ loop measure can be obtained. To date, only a few results on the integrability of the $\SLE_\kappa$ loop measure are known, including its two- and three-point Green functions and its electrical thickness~\cite{ang2024imdozz,ang2024sle}, by Ang, Sun, Wu and the first author of this paper.

\appendix
\section{Carath\'eodory kernel theorem and its extension}

Let $\{U_n\}$ be a sequence of open subsets of $\C$ containing $z_0\in\C$. For each $n$, define $V_n$ to be the connected component (containing $z_0$) of the interior of the intersection $U_n \cap U_{n+1} \cap \cdots$. The \emph{kernel} of the sequence $\{U_n\}$ is the union of all such $V_n$ (it is either a connected open set containing $z_0$ or the single point set $\{z_0\}$). Let $U$ be an open subset of $\C$ containing $z_0$. We say $(U_n,z_0)$ converges to $(U,z_0)$ in the Carath\'eodory sense if $U$ is the kernel of every subsequence of $\{U_n\}$.

Note that if $U_n\to U$ in Hausdorff distance, and there exists $r>0$ such that $B(z_0,r)\subset U_n$ for all $n$, then $(U_n,z_0)\to (U,z_0)$ in the Carath\'eodory sense.

The following theorem is due to Carath\'eodory.
\begin{theorem}[Carath\'eodory kernel theorem]\label{thm:Caratheodory}

Let $f_n(z)$ be the conformal map defined on $\D$ with $f_n(0) = 0$ and $f_n'(0) > 0$. Then $f_n$ converges to a function $f$ locally uniformly if and only if $(f_n(\mathbb{D}),0)$ converges to $(U,0)$ for some open subset $U\neq\mathbb{C}$ in the Carath\'eodory sense. If $U=\{0\}$, then $f=0$; Otherwise, $U$ is simply connected and $f$ is the conformal map from $\D$ to $U$ with $f(0)=0,f'(0)>0$.

\end{theorem}

We extend Theorem~\ref{thm:Caratheodory} to the case when $(f_n)$ is normalized at three boundary points, which is needed in the proof of Lemma~\ref{lem:dH}.

\begin{proposition}\label{prop:bdy-caratheodory}
Let $\Gamma$ be a (relatively) open subset of $\S^1$.
Suppose that $\{U_n\}$ is a sequence of decreasing Jordan domains such that $\Gamma\subset\partial U_n$ for all $n$, and converges to a Jordan domain $U$ in Hausdorff distance with $\Gamma\subset\partial U$. Let $(a_i)_{1\le i\le 3}$ (resp. $(\alpha_i)_{1\le i\le3}$ be three distinct points on $\Gamma$ (resp. $\S^1$). Let $f_n:\D\to U_n$ (resp. $f:\D\to U$) be the conformal map such that $f_n(\alpha_i)=a_i$ (resp. $f(\alpha_i)=a_i$) for $1\le i\le3$. Then $f_n\to f$ locally uniformly.
\end{proposition}
\begin{proof}
Let $z_0\in U$. Then $(U_n,z_0)\to(U,z_0)$ in the Carath\'eodory sense. Let $g_n:\D\to U_n$ be a conformal map such that $g_n(0)=z_0$ and $g_n'(0)>0$. By Carath\'eodory kernel theorem, $g_n$ converges to $g$ locally uniformly; here $g$ is the conformal map from $\D$ to $U$ with $g(0)=z_0$ and $g'(0)>0$.

We claim that $g_n^{-1}:U_n\to\D$ converges to $g^{-1}:U\to \D$ locally uniformly on $U$. Since Montel's theorem ensures that $(g_n^{-1})$ is a normal family on $U$, it suffices to show $w_n:=g_n^{-1}(z)\to w:=g^{-1}(z)$ for every $z\in U$. For any $0<r<\dist(w,\partial\D)$, note that $g(B(w,r))\supset B(z,\frac{1}{4}|g'(w)|r)$ by Koebe’s 1/4 theorem. Since $g_n\to g$ locally uniformly and $g_n'(w)\to g'(w)$, using Carath\'eodory kernel theorem on $B(w,r)$, we find $(g_n(B(w,r)),w)\to(g(B(w,r)),w)$ in the Carath\'eodory sence. In particular, there exists $N$ such that $g_n(B(w,r))\supset B(z,\frac{1}{8}|g'(w)|r)$ for $n>N$. Since $g_n$ is univalent and $z=g_n(w_n)$, we conclude $w_n\in B(w,r)$ for $n>N$. Since $r$ is arbitrarily small, the claim then follows.

Next, since $\Gamma$ is a relatively open subset of $\S^1$, by Schwarz reflection, we further claim that $g_n^{-1}(a_i)\to g^{-1}(a_i)$ for $1\le i\le3$. Indeed, since $g(\frac{1}{2}\D)\cap\Gamma=\emptyset$, there exists $\varepsilon\in(0,\frac{1}{10})$ such that $B(a_i,2\varepsilon)\cap g(\frac{1}{2}\D)=\emptyset$. The local uniform convergence of $g_n\to g$ on $\D$ then implies $B(a_i,2\varepsilon)\cap g_n(\frac{1}{2}\D)=\emptyset$ for $n$ sufficiently large. Now using Schwarz reflection, we extend $g_n^{-1}$ to be analytic functions defined on $U':=U\cup(\cup_{i=1}^3B(a_i,\eps))$, and $g_n^{-1}(U')\subset 2\D$. Hence $g_n^{-1}$ is a normal family on $U'$, and every limit of its subsequence is equal to $g^{-1}$ when restricting on $U$. By uniqueness of analytic function, $g_n^{-1}$ defined on $U'$ converges locally uniformly to $g^{-1}$ defined on $U'$. In particular, we have $g_n^{-1}(a_i)\to g^{-1}(a_i)$ for $1\le i\le3$.

Note that $h_n:=g^{-1}_n\circ f_n$ (resp.\ $h:=g^{-1}\circ f$) is a conformal automorphism on $\D$, mapping $\alpha_i$ to $g_n^{-1}(a_i)$ (resp.\ $g^{-1}(a_i)$) for $1\le i\le3$. Hence $h_n(\alpha_i)\to h(\alpha_i), 1\le i\le3$ according to the above paragraph, which implies $h_n\to h$ locally uniformly.
Now, for any $z\in\D$,
\[
|f_n(z)-f(z)|=|g_n\circ h_n(z)-g\circ h(z)|\le |g_n(h_n(z))-g_n(h(z))|+|g_n(h(z))-g(h(z))|,
\]
which tends to $0$ as $n\to\infty$ (using the local uniform convergence of $g_n$). This shows that $f_n\to f$ locally uniformly since $\{f_n\}$ is a normal family.
\end{proof}

\section{Proof of (\ref{eq:loop_soup})}\label{appd2}
Now, we complete the proof of \eqref{eq:loop_soup}.
This can be justified as follows. Let $\mathcal{L}_\C^\cW$ be a whole-plane $\SLE_{8/3}$ loop soup of intensity $\frac{c}{2}$, and $(C_i)$ be the collection of its loop clusters surrounding $0$. For each $i$, denote by $\eta_i$ the inner boundary of $C_i$, i.e., the connected component of $\partial C_i$ that surrounds $0$ and is closest to $0$, which is a simple loop from the paragraph just above \cite[Theorem~2]{kemppainen2016nested}.
Below, we use $\cD(\gamma)$ to denote be the bounded connected component of $\C\setminus\gamma$ for a simple loop $\gamma\subset\C$.

We first claim that if we sample $\eta$
from the counting measure on $\{\eta_i\}$, then the joint law of $\eta$ and the loop configuration $\mathcal{L}_\eta$ inside $\cD(\eta)$ is the same as, an independent $\SLE_{8/3}$ loop soup in $\cD(\eta)$ of intensity $\frac{c}{2}$ with $\eta$ first sampled
from the intensity measure $\nu^i$ (defined for each measurable set $A$ of simple loops, $\nu^i(A)$ is the mean number of loops $\{\eta_i\}$ in $A$; see \cite{kemppainen2016nested})\footnote{Note that $\nu^i$ is indeed a constant multiple of $\SLE_\kappa^\lp$ restricted to loops surrounding $0$ by \cite[Theorem~2]{kemppainen2016nested}.}.
Based on this claim, by restricting the above to the event that $\eta$ is chosen to be the innermost loop of $(\eta_i)$ surrounding $D$ (given $\eta$, it corresponds to the event that $\mathcal{L}_\eta$ has no cluster surrounding $D$ without intersecting $D$), we are able to conclude \eqref{eq:loop_soup}.

It remains to show the claim. Suppose $V$ is any simply connected bounded domain containing $0$. The loops in $\{\eta_i\}$ (note that they all surround $0$ and are disjoint) that are contained in $V$ can be discovered in order from outside to inside. 
Hence, for any $k\ge1$, we can define $\eta_V^k$ as the $k$-th outermost loop in $V$\footnote{We mention that the proper definition of $\eta_V^k$ relies on the local finiteness property of the $\SLE_{8/3}$ loop soup clusters, which has been justified and exploited in~\cite[Section 2.2]{kemppainen2016nested}.}.
Note that conditioned on $\eta_V^k$, the loop soup configuration inside $\eta_V^k$ is an independent $\SLE_{8/3}$ loop soup on $\cD(\eta_V^k)$. This can be deduced from a standard approximation (see e.g.~\cite[Proof of Lemma 9.2]{sheffield2012conformal}). Namely, for any $n\ge1$, let $\mathcal{V}_n^k$ be the largest union of dyadic squares of side length $2^{-n}$ that are contained in $\cD(\eta_V^k)$. Then for each union $V_n$ of such dyadic squares, $\mathcal{L}_\C^\mathcal{W}\big|_{V_n}$ is independent of $\{\mathcal{V}_n^k=V_n\}$. Hence, conditionally on $\eta_V^k$, $\mathcal{L}_\C^\mathcal{W}\big|_{\mathcal{V}_n^k}$ is distributed as a $\SLE_{8/3}$ loop soup in $\mathcal{V}_n^k$. Since this holds for all $n$, the statement for $\eta_V^k$ follows.

Now, since the above statement holds for every $k$, it implies that if $\eta$ sampled from the counting measure on $\{\eta_V^k\}_{k\ge1}$, then the loop configuration inside $\cD(\eta)$ is an independent $\SLE_{8/3}$ loop soup in $\cD(\eta)$. Finally, the claim follows by letting $V\uparrow\C$ (note that the law of the loop sampled from the counting measure on $\{\eta_V^k\}_{k\ge1}$ then converges to the law of the loop sampled from the counting measure on $\{\eta_i\}$, which is exactly $\nu^i$).

\medskip
\textbf{Acknowledgement.} 
G.C. thanks Xin Sun and Baojun Wu for helpful comments and discussions. Y.G. thanks Wei Qian for bringing his attention to this problem and helpful discussions. We thank Wendelin Werner for pointing out the relevant work \cite{werner2003sles} and his remarks on the applicability of our results to the reversibility and duality properties of SLE.
We are grateful to an anonymous referee for 
numerous and valuable suggestions on previous versions of this manuscript.
Part of YG's work was done while working at City University of Hong Kong.
G.C. is supported by National Key R\&D Program of China (No.\ 2021YFA1002
700). Y.G. is supported by National Key R\&D Program of China (No.\ 2023YFA1010700).

\bibliographystyle{alpha}
\footnotesize{\bibliography{references}}

\newcommand{\etalchar}[1]{$^{#1}$}
\begin{thebibliography}{ACSW24b}

\bibitem[ACSW24a]{ang2024imdozz}
Morris Ang, Gefei Cai, Xin Sun, and Baojun Wu.
\newblock {Integrability of Conformal Loop Ensemble: Imaginary DOZZ Formula and
  Beyond}.
\newblock {\em arXiv preprint arXiv:2107.01788}, 2024.

\bibitem[ACSW24b]{ang2024sle}
Morris Ang, Gefei Cai, Xin Sun, and Baojun Wu.
\newblock {SLE Loop Measure and Liouville Quantum Gravity}.
\newblock {\em arXiv preprint arXiv:2409.16547}, 2024.

\bibitem[AM01]{airault2001unitarizing}
H\'el\`ene Airault and Paul Malliavin.
\newblock Unitarizing probability measures for representations of {V}irasoro
  algebra.
\newblock {\em J. Math. Pures Appl. (9)}, 80(6):627--667, 2001.

\bibitem[ARS25]{ang2022moduli}
Morris Ang, Guillaume Remy, and Xin Sun.
\newblock The moduli of annuli in random conformal geometry.
\newblock {\em Ann. Sci. \'Ec. Norm. Sup\'er. (4)}, 58:1037--1087, 2025.

\bibitem[BD16]{benoist2016sle}
St{\'e}phane Benoist and Julien Dub{\'e}dat.
\newblock An {SLE}$_2$ loop measure.
\newblock {\em Ann. Inst. H. Poincar\'e Probab. Statist.}, 52(3):1406--1436,
  2016.

\bibitem[BJ24]{baverez2024cft}
Guillaume Baverez and Antoine Jego.
\newblock {The CFT of SLE loop measures and the Kontsevich--Suhov conjecture}.
\newblock {\em arXiv preprint arXiv:2407.09080}, 2024.

\bibitem[Cam17]{camia2017scaling}
Federico Camia.
\newblock Scaling limits, {B}rownian loops and conformal fields.
\newblock {\em Advances in disordered systems, random processes and some
  applications}, pages 205--269, 2017.

\bibitem[Car06]{cardy2006n}
John Cardy.
\newblock The {${\rm O}(n)$} model on the annulus.
\newblock {\em J. Stat. Phys.}, 125(1):1--21, 2006.

\bibitem[CLQ{\etalchar{+}}25]{clqsw}
Gefei Cai, Jiaqi Liu, Wei Qian, Xin Sun, and Baojun Wu.
\newblock {In preparation}.
\newblock 2025.

\bibitem[CW24]{carfagnini2024onsager}
Marco Carfagnini and Yilin Wang.
\newblock Onsager-{M}achlup functional for {${\rm SLE}_{\kappa }$} loop
  measures.
\newblock {\em Comm. Math. Phys.}, 405(11):258, 2024.

\bibitem[Dub05]{MR2118865}
Julien Dub{\'e}dat.
\newblock {${\rm SLE}(\kappa,\rho)$} martingales and duality.
\newblock {\em Ann. Probab.}, 33(1):223--243, 2005.

\bibitem[Dub09]{dubedat2009duality}
Julien Dub\'edat.
\newblock Duality of {S}chramm-{L}oewner evolutions.
\newblock {\em Ann. Sci. \'Ec. Norm. Sup\'er. (4)}, 42(5):697--724, 2009.

\bibitem[FL13]{field2013reversed}
Laurence~S. Field and Gregory~F. Lawler.
\newblock Reversed radial {SLE} and the {B}rownian loop measure.
\newblock {\em J. Stat. Phys.}, 150(6):1030--1062, 2013.

\bibitem[GLQ26]{gao2022multiple}
Yifan Gao, Xinyi Li, and Wei Qian.
\newblock Multiple points on the boundaries of {B}rownian loop-soup clusters.
\newblock {\em Ann. Probab.}, 54:216--268, 2026.

\bibitem[GQW25]{gordina2025infinitesimal}
Maria Gordina, Wei Qian, and Yilin Wang.
\newblock {Infinitesimal conformal restriction and unitarizing measures for
  Virasoro algebra}.
\newblock {\em J. Math. Pures Appl.}, 195:103669, 2025.

\bibitem[KS07]{kontsevich2007malliavin}
Maxim Kontsevich and Yuri Suhov.
\newblock On {M}alliavin measures, {SLE}, and {CFT}.
\newblock {\em Tr. Mat. Inst. Steklova}, 258:107--153, 2007.

\bibitem[KW04]{kw04}
Richard~W. Kenyon and David~B. Wilson.
\newblock Conformal radii of loop models.
\newblock 2004.

\bibitem[KW16]{kemppainen2016nested}
Antti Kemppainen and Wendelin Werner.
\newblock The nested simple conformal loop ensembles in the {R}iemann sphere.
\newblock {\em Probab. Theory Related Fields}, 165(3-4):835--866, 2016.

\bibitem[Law05]{lawler2008conformally}
Gregory~F. Lawler.
\newblock {\em Conformally invariant processes in the plane}, volume 114 of
  {\em Mathematical Surveys and Monographs}.
\newblock American Mathematical Society, Providence, RI, 2005.

\bibitem[Law09]{MR2518970}
Gregory~F. Lawler.
\newblock Partition functions, loop measure, and versions of {SLE}.
\newblock {\em J. Stat. Phys.}, 134(5-6):813--837, 2009.

\bibitem[Law11]{lawler2011defining}
Gregory~F. Lawler.
\newblock {Defining SLE in multiply connected domains with the Brownian loop
  measure}.
\newblock {\em arXiv preprint arXiv:1108.4364}, 2011.

\bibitem[LSW01a]{LSW2001a}
Gregory~F. Lawler, Oded Schramm, and Wendelin Werner.
\newblock Values of {B}rownian intersection exponents. {I}. {H}alf-plane
  exponents.
\newblock {\em Acta Math.}, 187(2):237--273, 2001.

\bibitem[LSW01b]{LSW2001b}
Gregory~F. Lawler, Oded Schramm, and Wendelin Werner.
\newblock Values of {B}rownian intersection exponents. {II}. {P}lane exponents.
\newblock {\em Acta Math.}, 187(2):275--308, 2001.

\bibitem[LSW02a]{LSW2002b}
Gregory~F. Lawler, Oded Schramm, and Wendelin Werner.
\newblock Analyticity of intersection exponents for planar {B}rownian motion.
\newblock {\em Acta Math.}, 189(2):179--201, 2002.

\bibitem[LSW02b]{LSW2002a}
Gregory~F. Lawler, Oded Schramm, and Wendelin Werner.
\newblock Values of {B}rownian intersection exponents. {III}. {T}wo-sided
  exponents.
\newblock {\em Ann. Inst. H. Poincar\'e Probab. Statist.}, 38(1):109--123,
  2002.

\bibitem[LSW03]{lawler2003conformal}
Gregory Lawler, Oded Schramm, and Wendelin Werner.
\newblock Conformal restriction: the chordal case.
\newblock {\em J. Amer. Math. Soc.}, 16(4):917--955, 2003.

\bibitem[LW00]{MR1796962}
Gregory~F. Lawler and Wendelin Werner.
\newblock Universality for conformally invariant intersection exponents.
\newblock {\em J. Eur. Math. Soc.}, 2(4):291--328, 2000.

\bibitem[LW04]{lawler2004brownian}
Gregory~F. Lawler and Wendelin Werner.
\newblock The {B}rownian loop soup.
\newblock {\em Probab. Theory Related Fields}, 128(4):565--588, 2004.

\bibitem[LY21]{lawler2021new}
Gregory~F. Lawler and Stephen Yearwood.
\newblock {A new proof of reversibility of ${\rm SLE}_\kappa$ for
  $\kappa\le4$}.
\newblock {\em arXiv preprint arXiv:2111.06960}, 2021.

\bibitem[MS16a]{miller2016imaginary1}
Jason Miller and Scott Sheffield.
\newblock Imaginary geometry {I}: interacting {SLE}s.
\newblock {\em Probab. Theory Related Fields}, 164(3-4):553--705, 2016.

\bibitem[MS16b]{miller2016imaginary2}
Jason Miller and Scott Sheffield.
\newblock Imaginary geometry {II}: reversibility of {${\rm
  SLE}_\kappa(\rho_1;\rho_2)$} for {$\kappa\in(0,4)$}.
\newblock {\em Ann. Probab.}, 44(3):1647--1722, 2016.

\bibitem[MS16c]{miller2016imaginary3}
Jason Miller and Scott Sheffield.
\newblock Imaginary geometry {III}: reversibility of {$\rm SLE_\kappa$} for
  {$\kappa\in(4,8)$}.
\newblock {\em Ann. of Math. (2)}, 184(2):455--486, 2016.

\bibitem[NW11]{NW11}
{\c{S}}erban Nacu and Wendelin Werner.
\newblock Random soups, carpets and fractal dimensions.
\newblock {\em J. Lond. Math. Soc.}, 83(3):789--809, 2011.

\bibitem[Qia18]{qian2018conformal}
Wei Qian.
\newblock Conformal restriction: the trichordal case.
\newblock {\em Probab. Theory Related Fields}, 171(3-4):709--774, 2018.

\bibitem[Qia21]{qian2021generalized}
Wei Qian.
\newblock Generalized disconnection exponents.
\newblock {\em Probab. Theory Related Fields}, 179(1-2):117--164, 2021.

\bibitem[Sch00]{Sc2000}
Oded Schramm.
\newblock Scaling limits of loop-erased random walks and uniform spanning
  trees.
\newblock {\em Israel J. Math.}, 118:221--288, 2000.

\bibitem[SSW09]{schramm2009conformal}
Oded Schramm, Scott Sheffield, and David~B. Wilson.
\newblock Conformal radii for conformal loop ensembles.
\newblock {\em Comm. Math. Phys.}, 288(1):43--53, 2009.

\bibitem[SW12]{sheffield2012conformal}
Scott Sheffield and Wendelin Werner.
\newblock Conformal loop ensembles: the {M}arkovian characterization and the
  loop-soup construction.
\newblock {\em Ann. of Math. (2)}, 176(3):1827--1917, 2012.

\bibitem[vdBCL16]{van2016random}
Tim van~de Brug, Federico Camia, and Marcin Lis.
\newblock Random walk loop soups and conformal loop ensembles.
\newblock {\em Probab. Theory Related Fields}, 166(1-2):553--584, 2016.

\bibitem[Wer03]{werner2003sles}
Wendelin Werner.
\newblock S{LE}s as boundaries of clusters of {B}rownian loops.
\newblock {\em C. R. Math. Acad. Sci. Paris}, 337(7):481--486, 2003.

\bibitem[Wer05]{werner2005conformal}
Wendelin Werner.
\newblock Conformal restriction and related questions.
\newblock {\em Probab. Surv.}, 2:145--190, 2005.

\bibitem[Wer08]{werner2008conformally}
Wendelin Werner.
\newblock The conformally invariant measure on self-avoiding loops.
\newblock {\em J. Amer. Math. Soc.}, 21(1):137--169, 2008.

\bibitem[Wu15]{wu2015conformal}
Hao Wu.
\newblock Conformal restriction: the radial case.
\newblock {\em Stochastic Process. Appl.}, 125(2):552--570, 2015.

\bibitem[WW13]{werner2013cle}
Wendelin Werner and Hao Wu.
\newblock From {${\rm CLE}(\kappa)$} to {${\rm SLE}(\kappa,\rho)$}'s.
\newblock {\em Electron. J. Probab.}, 18(36):1--20, 2013.

\bibitem[WX25]{wang2024brownian}
Yilin Wang and Yuhao Xue.
\newblock {The Brownian loop measure on Riemann surfaces and applications to
  length spectra}.
\newblock {\em Communications on Pure and Applied Mathematics},
  78(11):2123--2148, 2025.

\bibitem[Zha08a]{zhan2008duality}
Dapeng Zhan.
\newblock Duality of chordal {SLE}.
\newblock {\em Invent. Math.}, 174(2):309--353, 2008.

\bibitem[Zha08b]{zhan2008reversibility}
Dapeng Zhan.
\newblock Reversibility of chordal {SLE}.
\newblock {\em Ann. Probab.}, 36(4):1472--1494, 2008.

\bibitem[Zha21]{zhan2021sle}
Dapeng Zhan.
\newblock S{LE} loop measures.
\newblock {\em Probab. Theory Related Fields}, 179(1-2):345--406, 2021.

\end{thebibliography}

\end{document}